\documentclass[review,hidelinks,onefignum,onetabnum]{siamart250211}

\usepackage{lipsum}
\usepackage{amsfonts}
\usepackage{graphicx}
\usepackage{epstopdf}
\usepackage{algorithmic}
\usepackage{amsmath, amssymb}
\usepackage{mathrsfs} 
\usepackage{lipsum}
\usepackage{amsfonts}
\usepackage{graphicx}   
\usepackage{subcaption} 
\usepackage{epstopdf}
\usepackage{tikz}
\usetikzlibrary{decorations.pathmorphing,decorations.markings,arrows.meta}
\usetikzlibrary{arrows.meta}
\usepackage{accents}
\usepackage{algorithmic}
\usepackage{enumitem}
\setlist[enumerate]{leftmargin=.5in}
\setlist[itemize]{leftmargin=.5in}

\newsiamremark{remark}{Remark}
\newsiamremark{hypothesis}{Hypothesis}
\crefname{hypothesis}{Hypothesis}{Hypotheses}
\newsiamthm{claim}{Claim}
\newsiamremark{fact}{Fact}
\crefname{fact}{Fact}{Facts}
\ifpdf
  \DeclareGraphicsExtensions{.eps,.pdf,.png,.jpg}
\else
  \DeclareGraphicsExtensions{.eps}
\fi

\crefname{fact}{Fact}{Facts}
\DeclareMathAlphabet{\mathpzc}{OT1}{pzc}{m}{it}
\DeclareMathAlphabet{\mathcalligra}{T1}{calligra}{m}{n}
\usepackage{blindtext}
\newtheorem{exam}[theorem]{\bf Example}

\newcommand{\ds}{\displaystyle}

\newcommand{\be}{\begin{equation}}
	\newcommand{\ee}{\end{equation}}
\newcommand{\beno}{\begin{equation*}}
	\newcommand{\eeno}{\end{equation*}}
\newcommand{\ba}{\begin{align}}
	\newcommand{\ea}{\end{align}}
\newcommand{\bano}{\begin{align*}}
	\newcommand{\eano}{\end{align*}}
\newcommand{\bea}{\begin{eqnarray}}
	\newcommand{\eea}{\end{eqnarray}}
\newcommand{\beano}{\begin{eqnarray*}}
	\newcommand{\eeano}{\end{eqnarray*}}

\numberwithin{equation}{section}

\newcommand{\oldfactorial}[1]{\mathpalette\oldfactorialaux{#1}}
\newcommand{\oldfactorialaux}[2]{%
	{#1\mkern1mu\oalign{\vrule\,$#1#2\mathstrut$\,\cr\noalign{\hrule}}}}

\def  \argmin {{\tt argmin~}}

\def\sjump#1{[\hskip -1.5pt[#1]\hskip -1.5pt]}

\def \meas {{\rm meas}}
\def \max{{\rm max}}
\def \min{{\rm min}}
\def \sup{{\rm sup}}
\def \ta{{\mathtt{\bf a}}}
\def \esssup{{\rm esssup}}
\def \diam{{\rm diam}}
\def \supp{{\rm supp}}

\def \div {\mathrm{div}}

\def  \argmin {{\tt argmin~}}

\def\sjump#1{[\hskip -1.5pt[#1]\hskip -1.5pt]}

\def \meas {{\rm meas}}
\def \max{{\rm max}}
\def \min{{\rm min}}
\def \sup{{\rm sup}}
\def \ta{{\mathtt{\bf a}}}
\def \esssup{{\rm esssup}}
\def \diam{{\rm diam}}
\def \supp{{\rm supp}}
\def \dx{{\rm dx}}
\def \ds{{\rm ds}}
\headers{A convergent AFEM  for a phase-field model of dynamic fracture}{Ram Manohar and S.M. Mallikarjuaniah}

\title{A convergent adaptive finite element method for a phase-field model of dynamic fracture \thanks{Submitted to the editors DATE.
\funding{This work was funded by the National Science Foundation under grant  no.~2316905..}}}

\author{Ram Manohar \thanks{Department of Mathematics \& Statistics, Texas A \& M  University Corpus-Christi, TX-78412, USA (\email{ram.manohar@tamucc.edu}).}
	\and  S.M. Mallikarjuaniah\thanks{Department of Mathematics \& Statistics, Texas A \& M  University Corpus-Christi, TX-78412, USA
		(\email{M.Muddamallappa@tamucc.edu}).}
}

\usepackage{amsopn}

\ifpdf
\hypersetup{
  pdftitle={A convergent AFEM for a phase-field model of dynamic fracture},
  pdfauthor={Ram Manohar and S.M. Mallikarjuaniah}
}
\fi

\begin{document}

\maketitle

\begin{abstract}
We propose and analyze an adaptive finite element method for a phase-field model of dynamic brittle fracture. The model couples a second-order hyperbolic equation for elastodynamics with the Ambrosio–Tortorelli regularization of the Francfort–Marigo variational fracture energy, which circumvents the need for explicit crack tracking. Our numerical scheme combines a staggered time-stepping algorithm with a variational inequality formulation to strictly enforce the irreversibility of damage. The mesh adaptation is driven by a residual-based \textit{aposteriori-type} estimator, enabling efficient resolution of the evolving fracture process zone. The main theoretical contribution is a rigorous convergence analysis, where we prove that the sequence of discrete solutions generated by the AFEM converges (up to a tolerance) to a critical point of the governing energy functional. Numerical experiments for a two-dimensional domain containing an edge-crack under dynamic anti-plane shear loading demonstrate our method's capability of autonomously capturing complex phenomena, including crack branching and tortuosity, with significant computational savings over uniform refinement.
\end{abstract}



\section{Introduction}
The variational approach of \textit{Francfort and Marigo}~\cite{Francfort1998} treats brittle fracture as a minimization of total elastic and surface energy.
While formalized in the computationally challenging $\widehat{\mathrm{SBV}}$ space, this problem is rendered tractable by phase-field approximations, which regularize sharp cracks with a smooth scalar field.Building on this, a rigorous framework for dynamic fracture was established using a time-discrete scheme based on the Ambrosio--Tortorelli approximation~\cite{bourdin2011time, larsen2010models, larsen2010existence}.
This scheme alternates between solving a wave equation for displacement and minimizing an energy functional for the phase-field variable.
The resulting discrete trajectories are proven to converge to a weak solution of the continuous problem, satisfying an irreversibility constraint while upholding the physical principles of elastodynamics, energy balance, and maximal dissipation~\cite{Giacomini2005,larsen2010models}.
To efficiently solve fracture problems, where cracks are highly localized and the corresponding phase-field variable exhibits steep gradients, {Adaptive Finite Element Methods (AFEM)} have become an indispensable computational tool. A uniform mesh fine enough to capture the crack's small length scale everywhere would be computationally prohibitive. AFEM elegantly circumvents this by dynamically refining the computational mesh only in critical regions---primarily around the advancing crack tip---while leaving the mesh coarse elsewhere. This targeted approach ensures both accuracy and efficiency. This adaptive strategy is most often applied to phase-field models based on the {Ambrosio-Tortorelli regularization}~\cite{Ambrosiott1990, Ambrosiottapprox1992}, which transforms the sharp fracture problem into a more tractable sequence of coupled non-linear elliptic boundary value problems that can be discretized using the finite element method~\cite{fernando2025xi,  manohar2025adaptive, manohar2025convergence}. The theoretical foundation for applying AFEM to these models was largely established by Burke and collaborators~\cite{burke2010adaptive, burke2010BurkeOrtnerEndre}. Their seminal work introduced provably convergent adaptive algorithms for quasi-static brittle fracture by creating a rigorous feedback loop between \textit{a posteriori-type estimation} and targeted \textit{mesh refinement} (improving resolution in those specific areas). This methodology was subsequently extended to the generalized Ambrosio--Tortorelli functional to accommodate broader energy densities~\cite{burke2013adaptive}. The need for adaptivity becomes even more acute in the context of \textit{dynamic fracture}. As cracks can propagate rapidly and unpredictably, a responsive, adaptive mesh is not just a matter of efficiency but a necessity for accurately capturing the evolving physics of the failure event. \\

Let $\Omega:=\Omega(t) \subset \mathbb{R}^{2},$ be a smooth, open, connected, bounded domain with a given boundary $\partial \Omega$ and $\Gamma(t)$ is a crack set. For $T_f(\in \mathbb{R})<\infty$, the variable $t\in[0,T_f]$,  fixedis the time and it only enters through the time-dependent loading conditions and $T_f$ is the final time. Set $\Omega_T=\Omega\times [0, T_f]$. We assume that the discontinuity set $\Gamma(t)$ is completely contained within $\Omega(t)$, and $\Gamma(t)$ is a Hausdorff measurable set. For $\kappa >0$ and $\epsilon >0$, we define the regularized elastic energy  $E \colon \mathbb{U}\times \mathbb{V}\to \mathbb{R}^+_\infty$, where $\mathbb{R}^{+}_{\infty}=\mathbb{R}\cup \{+\infty\}$, 
and the regularized crack-surface energy $\mathcal{H} \colon \mathbb{V} \to \mathbb{R}$,
respectively, as
\begin{align}\label{minmization}
	E(u, \, v):=\frac{\mu}{2} \int_{\Omega} \ta(t)\,|\nabla u|^2\,\dx, \quad \text{and} \quad
	\mathcal{H}(v):=\int_{\Omega} \left[ \frac{(1-v)}{\epsilon} + \epsilon \; |\nabla v|^2  \right]  \; dx,
\end{align}
where $\mathbb{V}=\mathbb{U}=H^1(\Omega)$ and $\ta(t)=(1-\kappa)\,(v(t))^2+\kappa$. Here, $\kappa \ll 1$ is a numerical regularization parameter for the bulk energy term. Then, the total energy is given by the functional  $\mathcal{J}_{\epsilon}: \mathbb{U}\times \mathbb{V} \longmapsto \mathbb{R}$ such that
\begin{align}
	\mathcal{J}_{\epsilon}(u, \, v) &:=   E(u, \, v) + \frac{\lambda_c}{c_w} \;  \mathcal{H}(v) \notag \\
	&=\frac{\mu}{2}\int_{\Omega}\ta(t)\,|\nabla u|^2\, \dx +\frac{\lambda_c}{c_w} \;  \,\int_{\Omega}\Big[\frac{(1-v)}{\epsilon} + \epsilon \, |\nabla v|^2 \Big] \, \dx. \label{reg:energy}
\end{align}
At each time step, the energy associated with the phase-field function $v$ is minimized, adhering to irreversibility constraints. At the same time, the displacement field $u$ evolves dynamically, with its behavior significantly influenced by $v$ only within the small region representing the crack, allowing for discontinuities in $u$ within this region. More precisely, the dynamics is governed by the following, for $t\in (0, T_f]$, 
\begin{subequations}
	\begin{align}
		\varrho\,u_{tt}- \div \big(\ta(t)\,(\mu\,\nabla u+\eta\,\nabla u_t )\,\big)=f(x,t), \quad \;\; (x,t)\in \Omega \times [0, T_f] &\label{contwave}\\
		u(x, 0) =u_0(x), \quad \text{and} \quad 
		u_t(x, 0) =u_1(x), \quad \;\; x\in \Omega,&\label{contcond}\\
		v(t)\; \text{minimizes} \;\; \tilde{v} \mapsto \mathcal{J}_{\epsilon}(u; \, \tilde{v}) \quad \text{with} \quad  \tilde{v} \leq v(t),\;\;v\downarrow \text{as}~~ t \downarrow,
	\end{align}   
\end{subequations} 
Here, $u_t$ and $u_{tt}$ denote the first and second-order partial derivatives with respect to $t$, respectively. Note that we first solve $v(x, 0)$ minimizing $v \mapsto \mathcal{J}_{\epsilon}(u_0; \, v)$, where  $v$ represents a scalar \textit{phase-field function}.  \smallskip

Our aim is to study the \textit{dynamic-fracture} evolution in the brittle material under the action of time-varying load $g(t) \in \mbox{L}^{\infty}([0, \, T_f]; W^1_\infty(\Omega)) \cap \mbox{W}^1_1([0, \, T_f]; \mbox{H}^{1}(\Omega))$ applied on an open subset $\Omega_D \subset \Omega$. To simplify notation for boundary conditions, we set up the following 
\begin{equation} 
	\mathrm{Q}(g(t)) :=\left\{ u \in \widehat{\mathrm{SBV}}(\Omega) ~\colon~~ u|_{\Omega_D} = g(t) \right\}.
\end{equation}
Since functions of bounded variation might have discontinuities that are represented in their distributional gradient $\mathrm{D}w$, and assembled in the form of
\begin{align*}
	\mathrm{D} w~=~ \nabla w \mathrm{L}^d+ (w^+(x)-w^-(x)) \otimes \nu_{w}(x) \mathrm{H}^{d-1}\,\oldfactorial \,\mathrm{J}(w)+\mathrm{D}^c w,
\end{align*}
where $\nabla w$ and $\mathrm{J}(w)$ denote the approximate gradient of $w$ and the jump set of $w$, while $w^\pm$  denotes the inner and outer traces of $w$ on $\mathrm{J}(w)$ with respect to $\nu_w$ (the unit normal vector to $\mathrm{J}(w)$). In addition, $\mathrm{D}^c(w)$ represents the Cantor part of the derivative. However,  $\mathrm{H}^d$ and $\mathrm{L}^d$, respectively, represent the $d$-dimensional Hausdorff and Lebesgue measures. For a better understanding of this notation, we refer to \cite{Ambrosio2002}. 

This paper presents a complete framework for the adaptive finite element method for the approximation of the dynamic fracture. We begin in \Cref{sec2:varfor} by defining the necessary function spaces for the model's variational formulation. \Cref{sec3:modeldiscre} then describes the numerical approach, covering the time-discretization and finite element methods. The core of our contribution is in \Cref{sec4:convalg}, where we develop an adaptive mesh refinement algorithm guided by a novel residual-type error estimator and prove its convergence as established in \Cref{Cngwithtol}. To demonstrate the model's predictive power, \Cref{expres} presents numerical results for a complex fracture scenario. We conclude in \Cref{sec6:conc} with a summary and final remarks.

\section{Variational formulation}\label{sec2:varfor}
This section provides relevant function spaces and a variational formulation, which are essential to develop the groundwork for the forthcoming analysis.

For a positive integer $m\geq 0$, let $ \mbox{W}^m_p(\Omega)$ be the Sobolev space of order $(m,p)$ (cf., \cite{adams1975}), abbreviated by 
\begin{equation*}
	\mbox{W}^m_p(\Omega)~:=\big\{\varphi \in \mbox{L}^p(\Omega):\quad D^l \varphi \in \mbox{L}^p(\Omega), |l|\leq m\big\}.
\end{equation*}
for $1\leq p \leq \infty$, equipped with inner product and the norm
\begin{equation*}
	(\varphi, \psi)_{m,p,\Omega}~=~  \sum_{\iota \leq m}\int_\Omega  D^l \varphi \cdot D^l \psi\ , \dx,   \quad 
	\text{and} \quad  
	\|\varphi\|_{\mbox{W}^m_p(\Omega)}~=~ \Big(\sum_{|l| \leq m}\int_\Omega  |D^l \varphi|^p\, \dx\Big)^{1/p},   
\end{equation*}
respectively. For $p=\infty$, the norm is given by
\begin{equation*}
	\|\varphi\|_{\mbox{W}^m_\infty(\Omega)}~=~ \esssup_{|l| \leq m}\|D^l \varphi\|_{L^\infty(\Omega)}.   
\end{equation*}
Moreover, we write $\mbox{W}^m_2(\Omega):=\mbox{H}^m(\Omega)$ for $p=2$, and the norm is denoted by $\|\cdot\|_{\mbox{H}^m(\Omega)}$. 

Further, let $\mbox{L}^r([0,T_f];\mbox{W}^m_p(\Omega))$ be the Banach space of all $\mbox{L}^r$-integrable functions from $[0,T_f]$ to  $\mbox{W}^m_p(\Omega)$ such that 
$$\mbox{L}^r([0,T_f];\mbox{W}^m_p(\Omega)):=\Big\{\varphi:~ \int_0^{T_f} \|\varphi\|^r_{\mbox{W}^m_p(\Omega)}< \infty\Big\}$$
and the norm on $\mbox{L}^r([0,T_f];\mbox{W}^m_p(\Omega))$  is denoted by
$$\|\varphi\|_{\mbox{L}^r([0,T_f];\mbox{W}^m_p(\Omega))}=\Big(\int_0^{T_f}\|\varphi\|^r_{\mbox{W}^m_p\left(\Omega\right)}\,dt\Big)^{\frac{1}{r}}.$$
Furthermore, let $X$ be a Banach space, and for $l>0$, we define Bochner space $\mbox{C}^l([0,T_f];\mbox{X})$, as 
\begin{align*}
\mbox{C}^l([0,T_f];\mbox{X}) :=\big\{\, u:[0,T_f]\to \mbox{X} \ \big|\   \text{the} ~& \text{$i$-th (Bochner) derivative } u^{(i)} \text{ exists and } \\ 
&~~~ u^{(i)}\in \mbox{C}([0,T_f]; \mbox{X})\ \text{for }i=0,1,\dots,l \,\big\}
\end{align*}
associated with norm on \(\mbox{C}^l([0,T_f]; \mbox{X})\) as 
\begin{align*}
\|u\|_{\mbox{C}^l([0,T_f]; \mbox{X})}
:=\sum_{i=0}^{l}\sup_{t\in[0,T_f]}\|u^{(i)}(t)\|_{\mbox{X}}.
\end{align*}

For a fixed time $t=t_n$, we define 
\begin{subequations}
	\begin{eqnarray}
		&&\mathbb{U}_d:=\big\{\psi \in \mathbb{U}\,:~~ \psi=0~~ \text{on}~~ \Omega_D \big\},   \\
		&& \mathbb{U}_g:=\big\{\psi \in \mathbb{U}\,:~~ \psi=g(t_n)~~ \text{on}~~ \Omega_D \big\},\\
		\text{and}~~~ && \nonumber\\
		&& \mathbb{V}_c:=\big\{\varphi \in \mathbb{V}\,:~~ \varphi=0 ~~ \text{on}~~ CR(t_{n-1}) \big\},
	\end{eqnarray}
\end{subequations}
respectively. Further, we set $\mathbb{V}_c^\infty=\mathbb{V}_c\cap \mbox{L}^\infty(\Omega)$.

 For given $f\in \mbox{L}^2([0,T_f]; \mbox{H}^1(\Omega))$, $u_0(x)\in \mbox{H}^2(\Omega)$, $u_1(x)\in \mbox{H}^1(\Omega)$,  we define the function spaces 
\begin{align*}
&\widehat{\mathbb{U}}(0,T_f) :=\mbox{C}([0,T_f];\mbox{H}^2(\Omega))\cap \mbox{C}^1([0,T_f]; \mbox{H}^1(\Omega))\cap \mbox{C}^2([0,T_f];\mbox{L}^2(\Omega)),\\ 
&\widehat{\mathbb{V}}(0,T_f) := \mbox{L}^\infty([0,T_f]; \mathbb{V}_c^\infty) \cap \mbox{W}^1_\infty([0,T_f]; \mbox{L}^2(\Omega)),
\end{align*}
respectively.
It is observed that  $$\widehat{\mathbb{U}}(0,T_f) \subset \mbox{C}([0,T_f]; \mathbb{U}_g)\quad \text{and} \quad \widehat{\mathbb{V}}(0,T_f)  \subset \mbox{L}^\infty([0,T_f]; \mathbb{V}_c^\infty).$$

By fixing parameter constants $\kappa$, $\epsilon$, $\lambda_c$ and $c_w$ for $t\in [0, T_f]$,  one may express  the functional $\mathcal{J}$ from $ \mathbb{U}_g\times \mathbb{V}_c$ to $\mathbb{R}^{+}_{\infty}$ as
\begin{align}
	\mathcal{J}(u, \, v)~=~\frac{1}{2}\int_{\Omega}\Big[ \mu\,\ta(t) \,|\nabla u|^2+2\,\nu\, (1-v)+ \rho \; |\nabla v|^2 \Big] \; \dx,  \label{minmizationfunc1.7}
\end{align}
where $\rho=\frac{2\,\lambda_c\, \epsilon}{c_w}$ and $\nu = \frac{\lambda_c}{c_w\, \epsilon}$, respectively. 
	
Hence, our objective is to search a minimizer $(u, v)\in \widehat{\mathbb{U}}(0,T_f) \times \widehat{\mathbb{V}}(0,T_f)$ that satisfies the following problem in the weak sense, 
\begin{subequations}
	\begin{align}
		\varrho\,\big(u_{tt}, \psi\big)+ \mathcal{B}(\ta(t);u,\psi)&+\widehat{\mathcal{B}}(\ta(t);u,\psi)=\big(f,\psi), \quad~~~  \psi \in \mathbb{U}_d,\label{minu}\\
		u(x, 0) =u_0(x),  \quad \text{and}& \quad   u_t(x, 0) =u_1(x), \quad ~~~~~~~~~~~ x\in \Omega, \label{bdry} \\
		\mathcal{J}(u; v)=&\inf_{\overset{\tilde{v} \in \mathbb{V}_c}{\tilde{v}\leq v}} \mathcal{J}(u; \tilde{v}), \quad \label{minv}
	\end{align}   
\end{subequations}
where
\begin{equation}
\mathcal{B}(\ta(t);u,\psi):=\int_\Omega \mu\, \ta(t)\,\nabla u \cdot \nabla \psi \, \dx, \;\;  \widehat{\mathcal{B}}(\ta(t);u,\psi):=\int_\Omega \eta\, \ta(t)\,\nabla u_t \cdot \nabla \psi \, \dx
\end{equation}
respectively.
By invoking a truncation argument, we notice that any local minimizer $(u, v)$ of $\mathcal{J}$ necessarily satisfies $0\leq v(x) \leq 1$ all the time $x \in \Omega$. This implies that the test function space for $v$ should be $L^\infty(\Omega)$. Consequently, we define the test space for $v$ as $\mathbb{V}_{T}=\mathbb{V}\cap L^\infty(\Omega)$. Note that we minimize our functional with respect to the phase field function $v$ by treating the displacement field variable $u$ as fixed. For simplicity, we write $	\mathcal{J}'(z; v)(\varphi)=	\mathcal{J}_v(z; v)(\varphi)$, where $z$ denotes the solution of \cref{contwave}-\cref{contcond}.

Thus,  $z\in \widehat{\mathbb{U}}(0,T_f)$ solves the wave dynamic problem, then the directional derivative of $\mathcal{J}$ in direction $ \varphi \in \mathbb{V}_{T}$ is defined by, for $t\in [0,T_f]$, 
\begin{align}
	\mathcal{J}'(z; v)(\varphi):=~\int_{\Omega}\big[\rho\, \nabla v \cdot \nabla \varphi -\nu \, \varphi +\mu\,(1-\kappa)\,v\, \varphi\,|\nabla z|^2\big] \, \dx.\label{3.18jprim}
\end{align}
Furthermore, the differentiability of $\mathcal{J}$ is confirmed by the remainder term $$\frac{|\mathcal{R}(u; v, \varphi)|}{~\|\varphi\|_{\mathbb{V}}+\|\varphi\|_{\infty}} \to 0  \quad \text{as} \quad \|\varphi\|_{\mathbb{V}}+\|\varphi\|_{\infty}\rightarrow 0.$$
While the remainder term $\mathcal{R}$ is calculated by utilizing the following definition
$$\mathcal{R}(u; v, \varphi)~=~ \mathcal{J}(u; v+\varphi)-\mathcal{J}(u; v)- \mathcal{J}^\prime(u; v)(\varphi).$$ 
In addition, the functional $\mathcal{J}$ lacks G$\check{a}$territorial differentiation over the whole $\mathbb{U}\times \mathbb{V}$. 
To avoid this, we restricted our analysis to critical points within the subspace $\mathbb{V}_c^\infty$ of $\mathbb{V}_c$.
The following proposition shows that the  local minimizer $v$ of the functional  $\mathcal{J}$  satisfies the necessary condition $0\leq v(x,t)\leq 1$. For a proof, one may refer to   \cite{manohar2025adaptive,manohar2025convergence}. 
\begin{proposition} \label{prop3.2:v}
	Let $u \in \widehat{\mathbb{U}}(0,T_f) $  be the solution of the problem \cref{contwave}-\cref{contcond}. Further, we assume that $v\in  \widehat{\mathbb{V}}(0,T_f) $ is the minimizer of the functional \cref{reg:energy}, then $v(x,t)$ satisfies the condition $0 \leq v(x,t) \leq 1$ for {\it a.e.} $x\in \Omega$, $t\in [0,T_f]$.
\end{proposition}

Next we discretize our minimization problem utilizing the AFEM method in the following section. 
\section{Discrete version of the model} \label{sec3:modeldiscre}
This section describes the finite element setup for our model problem. \smallskip

Let $\mathscr{T}_{h} $  be a family of regular simplicial triangulations of the domain $\bar{\Omega}$ such that the boundaries of the triangles exactly represent the boundary $\Gamma$. 
Further, for any two distinct triangles $\tau_i,\, \tau_j \in \mathscr{T}_{h},\, i \neq j$, then their intersection is either empty, a vertex, an edge, or a $k-$dimensional face, where $0\leq k \leq d-1$. For each element $\tau \in \mathscr{T}_h$, we define its diameter as $h_\tau=\diam(\tau)$ and the mesh size $h=\max_{\tau \in \mathscr{T}_h} h_\tau$. In addition,ion, we set $h_{e}:=\diam(e)$.
Next, we define the set of all edges or $(d-1)$-dimensional faces by $\mathcal{E}_{h}$ such that $\mathcal{E}_{h}:=\mathcal{E}_{int, h} \cup \mathcal{E}_{bd,h}$, where the boundary edges $\mathcal{E}_{bd,h}$ are given by the Dirichlet boundary $\mathcal{E}_{D, h}:= \{e_j \in \mathcal{E}_{h}:~~e_j \subset \bar{\Omega}_D \}$, and the Neumann boundary $\mathcal{E}_{N, h}:= \{e_j \in \mathcal{E}_{h}:~~ e_j \subset \partial \Omega_N \},$ respectively, while the set of interior edges $\mathcal{E}_{int, h}$ is denoted by $\mathcal{E}_{int, h}:= \{e_j \in \mathcal{E}_{h} \backslash (\bar{\Omega}_{D,h}\cup \partial \Omega_{N,h})\}$.
Furthermore, we assume that the triangulation satisfies the mesh shape-regularity condition $\sup_{\tau \in \mathscr{T}_h} \frac{h_\tau}{\varrho_\tau} \leq c_{\varrho},$ where $\varrho_\tau$
denotes the largest diameter of the inscribed $d$-dimensional ball in $\tau$, and $c_{\varrho}$ is a positive constant. In addition, for each physical element $\tau \in \mathscr{T}_h$, there exists an affine mapping $F:\hat{\tau}\mapsto \tau$, where $\hat{\tau}$ is the reference simplex such that
$\hat{\tau}:= \big\{\hat{\mathtt{x}}~:~~ 0<\sum_{j=1}^d \hat{x}_j<1, \quad \text{for}\;\; \hat{x}_j >0 \; \text{and}\;\; j\in [1:d] \big\}.$ 
We define the index set $\mathcal{N}_h \subset \mathbb{N}$ for the vertices of $\mathscr{T}_h$. We denote the basis function by $\xi_i$, $i \in \mathcal{N}_h$, such that $\xi_i$'s are continuous piecewise linear functions, and $\xi_i(x_j)=\nu_{ij}$,   then we will use the following notation in our analysis
$$\mathcal{N}_{\Omega, h}=\big\{j \in \mathcal{N}_h~| ~~~ x_j \in \bar{\Omega} \big\}, \quad  
\omega_{\tau_j}:= \bigcup_{\overset{j\in \mathcal{N}_h}{x_j \in \bar{\tau}_j}} \omega_j, \quad \text{and} \quad  \omega_{e_j}:=\bigcup_{\overset{j\in \mathcal{N}_h}{x_j \in \bar{e}_j}}  \omega_j$$ respectively, where $\omega_j:= \supp(v_j)$, is the closure of the union of elements $\tau_j \in \mathscr{T}_h$ that have $x_j$ as the position of a vertex for $j\in \mathcal{N}_h$.

We now divide the time domain $[0, T_f]$ into $\tt{N}_T$-parts with considering the time nodal points $0=t_0<t_1<t_2<\ldots < t_{\tt{N}_T}=T_f$ such that $j_n=(t_{n-1}, t_n]$ and $k_n := t_n-t_{n-1}$, $n=1,\,2, \ldots,\, \tt{N}_T$, respectively. In addition, we associate a conforming shape-regular triangulation $\mathscr{T}^n_{h}$ of the domain $\Omega$ with $t=t_n,\, 0\leq n \leq \tt{N}_T.$ The meshes are assumed to be compatible in the sense that for any two consecutive meshes $\mathscr{T}^{n-1}_{h}$ and $\mathscr{T}^n_{h}$, $\mathscr{T}^n_{h}$
can be obtained from $\mathscr{T}^{n-1}_{h}$ by
locally coarsening some of its elements and then locally refining some (possibly other) elements. The finite element spaces
corresponding to $\mathscr{T}^n_{h}$
are given by
$\mathbb{V}^n_h$, $\mathbb{U}^n_{d,h}$, and $\mathbb{U}^n_{g, h}$, at time  level $t=t_n$, $n=1,\,2, \ldots,\, \tt{N}_T$,
\begin{align*}
	&\mathbb{V}^n_h:=\big\{\sum_{j\in \mathcal{N}_h} \Lambda_j\xi_j: ~~ \Lambda_j\in \mathbb{R} \big\},  \nonumber\\
	&\mathbb{U}^n_{d,h}:=\big\{\sum_{j\in \mathcal{N}_h} \Lambda_j\xi_j: ~~ \Lambda_j\in \mathbb{R},~~ \Lambda_j=0~~ \text{for all}~j \in \mathcal{N}_{d,h}\big\},~\text{and}\\
	&\mathbb{U}^n_{g, h}:=\big\{\sum_{j\in \mathcal{N}_h} \Lambda_j\xi_j: ~~ \Lambda_j\in \mathbb{R},~~ \Lambda_j=g(t_n, x_j)~~ \text{for all}~j \in \mathcal{N}_{d,h} \big\}, 
\end{align*}
respectively. For a given tolerance ${\it \Xi_{CR}}$, and for $\varphi_h \in \mathbb{V}^n_h$, define a discrete version of $CR(t_{n-1})$, by
\begin{eqnarray*}
	\mathcal{E}_{h, CR}^{n-1}:=\big\{ e_j \in \mathcal{E}_h\,\colon ~~ \varphi_h(x, t_{n-1})\leq {\it \Xi_{CR}},~~ \text{for all}~~ x\in \bar{e}_j \big\},
\end{eqnarray*}
such that $CR_h(t_{n-1}):=\bigcup_{e_j \in \mathcal{E}_{h, CR}^{n-1}} \bar{e}_{j}$. Hence, the finite element space $\mathbb{V}^n_{c,h}$ is defined as
\begin{eqnarray*}
	\mathbb{V}^n_{c,h}:= \big\{\varphi_h\in \mathbb{V}^n_h~:~~ \varphi_h(x)=0,~~ \text{for all}~~ x\in CR_h(t_{n-1})\big\}.
\end{eqnarray*}
Let $u_h~:~[0, T_f] \longrightarrow \mathbb{V}_h$ be the piecewise linear interpolant in time corresponding to the nodal values  $\{u_h^n\}$, $n=0,\,1,\,2, \ldots,\, \tt{N}_T$, such that 
$$u_h(t)=\frac{t-t_{n-1}}{k_n}\,u_h^n+\frac{t_n-t}{k_n}\,u_h^{n-1},\quad t\in j_n,\; n=1,\,2, \ldots,\, \tt{N}_T$$
We now define the first and second backward finite differences by
\begin{subequations}
	\begin{align}
		& \delta u_h^n:=\begin{cases}
			\frac{u_h^n- u_h^{n-1}}{k_n}, \quad n = [1:\tt{N}_T] \\ \smallskip
			\tilde{u}_h^0=:\pi_h^0u_1, \quad n=0,
		\end{cases} \quad \text{and} \quad 
		\delta^2 u_h^n:= \frac{\delta u_h^n-\delta u_h^{n-1}}{k_n}, \quad n = [1:\tt{N}_T],
	\end{align}
\end{subequations}
where $\pi_h^0$ is a suitable projection operator in the related finite element space. The discrete version of the minimizing problem can expressed as
\begin{subequations}
	\begin{align}
&~~~~~~~~\begin{cases} 
			\varrho\, \big(\delta^2 u_h^n, \psi_h \big)+ \mathcal{B}^n_h(\ta_h^{n-1};u_h^n,\psi_h)+\widehat{\mathcal{B}^n_h}(\ta_h^{n-1};u_h^n,\psi_h)=\big(f_h^n,\psi_h), \quad  \psi_h \in \mathbb{V}_h^n \label{minuh} \\ 
			~u_h^1:=u_h^0+k_1 \tilde{u}_h^0, \quad \text{and} \quad u_h^0:=\pi_h^0u_0, 
		\end{cases} \\
	&~~~~~~~~~~~~\mathcal{J}^n_h(u^n_h; \, v^n_h)=~ \frac{1}{2}\int_{\Omega}\Big[ \mu\, \ta_h^{n}\,|\nabla u_h^n|^2+ 2\,\nu\, (1-v_h^n)+ \rho \; |\nabla v_h^n|^2 \Big] \; \dx \label{disminmizationfunc4.1}
	\end{align}   
\end{subequations}
with 
\begin{equation}
\mathcal{B}^n_h(\ta_h^{n-1};u_h^n,\psi_h):=  \big(\mu\,\ta_h^{n-1}\,\nabla u_h^n, \nabla \psi_h\big), \;\; \widehat{\mathcal{B}^n_h}(\ta_h^{n-1};u_h^n,\psi_h):=  \big(\eta\, \ta_h^{n-1}\,\nabla \delta u_h^n, \nabla \psi_h\big)
\end{equation}
where $\ta_h^{n}=(1-\kappa)(v_h^{n})^2+\kappa$ and $f_h^n:=\frac{1}{k_n}\int_{j_n}f_h\,\dx$, respectively. In addition, $f_h$ is a suitable approximation or projection of $f$.

To determine the critical points, we define the  derivative of the approximated  functional  ${(\mathcal{J}^n_h)}$ as
\begin{align}   
	{(\mathcal{J}^n_h)}^{\prime}(z_h^n;v_h^n)(\varphi_h)~:=~\int_{\Omega}\Big[\rho\,\nabla v_h^n \cdot \nabla \varphi_h - \nu \, \varphi_h + \mu\,(1-\kappa)\,(v^n_h \, \varphi_h)\, |\nabla z_h^n|^2\Big] \, \dx,  \label{ddofJbfem} 
\end{align}   
where $z_h^n$ is a solution of \cref{minuh}.
\begin{definition}\label{def4.1} Let $v^n_h\in \mathbb{V}^n_{c,h}$ is said to be a critical point of  $\mathcal{J}_h^n$ if  $v^n_h$ satisfy the condition $(\mathcal{J}^n_h)^{\prime}(u_h^n; v_h^n) (\varphi_h)=0$, $\forall \varphi_h \in  \mathbb{V}^n_{c,h}$ and  $u_h^n\in \mathbb{U}^n_{g,h}$ which solve the system \cref{minuh}.
\end{definition}
The following proposition, which is based on \Cref{def4.1}, states the critical points of the functional $\mathcal{J}^n_h(\ast; \cdot)$. The comprehensive proof of the proposition is included in Appendix A.
\begin{proposition} \label{prop4.2:vh}
	Let $u_h^n \in \mathbb{U}^n_{g,h}$ be the solution of \cref{minuh}, and $v_h^n \in \mathbb{V}^n_{c,h}$ such that $(\mathcal{J}^n_h)^{\prime}(u_h^n; v_h^n) (\phi_h)=0,~ \forall \phi_h \in \mathbb{V}^n_{c,h}$, then $v^n_h$ satisfies the condition $0 \leq v^n_h(x) \leq 1$ for all $x\in \Omega$.
\end{proposition}
The approximate model problem can be stated as follows: To find $z_h^n\in \mathbb{U}^n_{g,h}$ which solves \cref{minuh}, we minimize the functional at $t=t_n, \, n=1,\,2,\, \ldots, \mathtt{N}_T$, 
\begin{equation}
	\mathcal{J}_h^n(z_h^n; v_h^n)=\underset{\bar{v}_h^n \in  \mathbb{V}^n_{c,h}}{\argmin} \mathcal{J}^n_h(z_h^n; \bar{v}^n_h). \label{minag4.3}
\end{equation}
It follows that, at $t=t_n, \, n=1,\,2,\, \ldots, \mathtt{N}_T$,
\begin{equation}
	v_h^n \in \argmin\big\{\mathcal{J}^n_h(z_h^n; \bar{v}^n_h)~:~~ \bar{v}_h^n \in \mathbb{V}^n_{c,h}, \;\;z_h^n\in \mathbb{U}^n_{g,h}\;\; \text{solves}\;\; \cref{minuh} \big\}. \label{minagjh4.4}
\end{equation}
The next section presents different strategies for conducting the minimization problem provided by equation \cref{minag4.3}.  
\section{Convergence Analysis} \label{sec4:convalg}
Our aim is to develop the adaptive algorithm  and analyze the convergence of sequences generated by the adaptive algorithm. However, we note that the adaptive algorithm refines the mesh utilizing the minimization technique subsequent to the completion of the monolithic minimization technique \cite[Sec 4]{manohar2025convergence}. To minimize the functional $\mathcal{J}$ on the infinite-dimensional space $\mathbb{U}_g\times \mathbb{V}_c$ at $t=t_n$.  For simplicity, we write $u(t_n):=u^n$ and $v(t_n):=v^n$ at the time level $t_n, \, n=1,\,2,\, \ldots,\, \mathtt{N_T}$. Although,
 we consider the tolerances $\Xi_v$, $\Xi_{v_n}$ and $\Xi_{RF}$ for stopping the minimization and refinement loops, respectively; mesh size $h_k=\max_{\tau \in \mathscr{T}_{h, k}} \diam(\tau)$ as the maximum diameter of elements in the $k$-th refinement level $\mathscr{T}_{h, k}$; and marking parameter $0 < \vartheta \leq 1$ to determine elements for refinement. A crucial part of the adaptive algorithm is the local refinement indicator $\mathpzc{R}_{\tau}$ in equation \cref{esteta4.11bb}. It informs the mesh refining operation in the following adaptive algorith, which is guided by certain parameters and serves to boost numerical solution accuracy.
\smallskip
\begin{algorithm}[htp]
	\caption{Adaptive Algorithm}
	\label{adapalg}
	\begin{algorithmic}
		\vspace{0.2cm}
		
		\STATE	\textbf{Step 1.~}{\tt Initialization:} Input crack field $v_0=1$~ and~ $\mathscr{T}_{h,0}$.  \vspace{0.2cm}
		
		\STATE	\textbf{Step 2.~}{\tt Start time loop:} For $t_n$, $n=1,\,2,\ldots,\, N_T.$ \\
		\hspace{2cm} Compute $u^n_j$ and $v^n_j$ at $t_n$, for $j=1,\,2,\ldots$ \\ \vspace{0.1cm}
		\hspace{2.5cm} Compute $u^n_j$  from \cref{minuh} fixing  $v_j^{n-1}$\\
		
		\hspace{2.5cm} Compute $v_j^n=\underset{\bar{v}\in \mathbb{V}_{c}}{{\argmin}} \{\mathcal{J}(u_j^n, \bar{v})\}$
		\\ \vspace{0.2cm}
		
		\hspace{2.0cm} \textbf{Check:} If $ \|v_j^{n}-v_j^{n-1}\|_{L^\infty(\Omega)}\geq \Xi_{v_n}$\\
		
		\STATE	\hspace{5cm} Repeat \textbf{\tt Step 2} \\
		
		\hspace{4.5cm} Else  \\
		
		\hspace{5cm} \textbf{\tt Break}; \\
		
		\hspace{4.5cm} End Else \\
		\hspace{3.5cm}  If $v^n_j \leq \Xi_v $\\
		\hspace{4.5cm} $v^n_j= 0$\\
		\hspace{4.5cm} Else if $ v_n^j\geq 1.0$\\
		\hspace{5cm} $v_n^j=1.0$\\
		\hspace{4.5cm} End Else if\\
		\hspace{4.0cm} End If \\
		\hspace{3.5cm} End If \\ \vspace{0.2cm}
		\STATE	\textbf{Step 3.} Set $u_j=u_j^n$ and $v_j=v_j^n$. \\ \vspace{0.2cm}
		\STATE	\textbf{Step 4.} If $\big(\sum_{\tau \in \mathscr{T}^n_{h_j}}|\mathpzc{R}_{\tau}(u_j, v_j)|^2\big)^{1/2}> \Xi_{RF}$, \\ \vspace{0.15cm} 
		
		\hspace{2cm} Determine a smallest subset $\mathcal{M}^n_{h_j}$ of $\mathcal{T}^n_{h_j}$ satisfying \\ \vspace{0.1cm} 
		\hspace{2.7cm} $\sum_{\tau \in \mathcal{M}^n_{h_j}}|\mathpzc{R}_{\tau}(u_j, v_j)|^2 \geq \vartheta \sum_{\tau \in \mathscr{T}^n_{h_j}}|\mathpzc{R}_{\tau}(u_j, v_j)|^2$ \\ \vspace{0.2cm} 
		\hspace{2.7cm} Refine the set  $\tau \in \mathcal{M}^n_{h_j}$, then generate  new mesh $\mathscr{T}^n_{h_{j+1}}$ (say)\\ \vspace{0.1cm} 
		\hspace{1.55cm} End If 
		
		\STATE	\textbf{Step 5.} Set $u_h(t_n)=u_j$, $v_h(t_n)=v_j$, \\ \vspace{0.15cm} 
		\hspace{1.5cm} Repeat the steps. \vspace{0.15cm} \\
	\end{algorithmic}
\end{algorithm}

In order to derive the \textit{aposteriori-type} indicator, we first state the following quasi-interpolation approximation lemma (cf., \cite{Verfurth1999}). 
\begin{lemma} \label{4.3lemmaapr}
	Let $\pi^c_h$ be the {\it quasi interpolants}, then, there exist positive constant  $c_{41}$ and $c_{42}$ the constants may depend on the shape-regularity parameter of the mesh $\tau$ but not on the mesh size, such that, for all elements $\tau_j \in \mathscr{T}_h$  and edges $e_j \in \mathcal{E}_h$, $\forall\;  j\in \mathcal{N}_h$ and  $m \in \{0, 1\}$, 
for $\varphi \in \mathbb{V}_c,$  we have 
	\begin{subequations}
	\begin{align}   \label{4.10bbapr}
			& \|\varphi-\pi^c_h \varphi\|_{H^m(\tau_j)} \leq c_{41}\,h^{1-m}_{\tau_j} \|\nabla \varphi\|_{L^2(\omega_{\tau_j})} \\
			\text{and} ~~~~~ & \nonumber\\
			& \|\varphi-\pi^c_h \varphi\|_{L^2(e_j)} \leq c_{42}\,h^{1/2}_{e_j} \|\nabla \varphi\|_{L^2(\omega_{e_j})},
		\end{align}
	\end{subequations}
	respectively.
\end{lemma}
In the following lemma, we state the gradient convergence of  a sequence $\{\nabla u_i^n\}_{i=1}^\infty$ in  $\nabla u$ in $L^2(\Omega)$ at each $t=t_n$ in the sense of maximum norm convergence of sequence $\{v_i^n\}_{i=0}^\infty$. The proof of the following lemma can be easily derived by utilizing the idealogy of \cite{FrancfortBourdi2000, HechtLozinski2018,Johnson2009,WheelerWickWollner2014}. We omit the proof of the following lemma. 
\begin{lemma} \label{ugdbound} 
	Let $u \in\widehat{\mathbb{U}}(0,T_f)$ and $u_i^n \in \mathbb{U}^n_{g,h}$ be the solutions of continuous problem \cref{contwave}-\cref{contcond}  and the  discrete problem  \cref{minuh}, respectively.
	Then, there exists a positive constant $c_{43}>0$, independent of $h$ and $k$, such that
	\begin{align}
	\underset{0 \leq n \leq \tt{N}_T}{\max} \|\nabla (u(t_n) - u_i^n)\| \leq c_{43} \big( h + k + \underset{0 \leq n \leq \tt{N}_T}{\sup}  \|v(t_n) - v_i^n\|_\infty \big).  
	\end{align}
	In particular, if $h\to 0$, $k\to 0$, and $ \underset{0 \leq n \leq \tt{N}_T}{\sup} \|v(t_n) - v_i^n\|_\infty \to 0$, then
	\begin{align}
	\underset{0 \leq n \leq \tt{N}_T}{\max} \|\nabla (u(t_n) - u_i^n)\| \to 0.   
	\end{align}
\end{lemma}

For error estimation purposes, we define the notion of a jump for any $\phi_h \in \mathbb{V}_h$ across an internal element edge/face $e_{ij}\in \mathcal{E}_{int,h}$ shared by elements $\tau_i$ and $\tau_j$ ($i>j$), as follows
\begin{align*}
	&\sjump{\nabla \phi_h}:=~|\nabla \phi_h|_{e_{ij}\cap \partial\tau_i}-|\nabla \phi_h|_{e_{ij}\cap \partial\tau_j}, \quad  (\text{on interior edges of element})\\
	 \text{and} & \\
	 & \sjump{\nabla \phi_h}|_{e_{j}}=:~|\nabla \phi_h \cdot \textbf{n}|_{e_j},\;\; e_j \in \partial \tau_j \cap \partial \Omega, \quad  (\text{on boundary edges of element})
\end{align*}
where $\textbf{n}$ is the outer unit normal vector. 

In the following lemma, we derive the \textit{aposteriori-type} estimate which serves as the key point in our adaptive approach.
\begin{lemma} \label{4.5jbddlemma}
	For all $\varphi \in \mathbb{V}_c$, we assume that $u^n_h\in \mathbb{U}^n_{g,h}$ is the solution of \cref{minuh}, and let $v^n_h\in \mathbb{V}^n_{c,h}$ be such that $(\mathcal{J}_h^n)^\prime(u^n_h; v_h^n)(\varphi_h) =0$ for all $\varphi_h \in \mathbb{V}^n_{c,h}$. Then there exists a constant $c_{44}>0$ such that the following estimate  
	\begin{align}
		|\mathcal{J}^\prime(u_h^n; v^n_h)(\varphi) | \leq~ c_{44}\,\big\{\mathpzc{R}_h \times\|\nabla \varphi\|\big\} \label{4.12lemmabdd} 
	\end{align}
	holds,  where $\mathpzc{R}_h$ is defined as $	\mathpzc{R}_h:=~\Big\{\sum_{ \tau_i\in \mathscr{T}_h^n}|\mathpzc{R}_{\tau}(u_h^n,v_h^n)|^2\Big\}^{\frac{1}{2}}$
	with
		\begin{align} 
			|\mathpzc{R}_{\tau}(u_h^n,v_h^n)|^2 &:=~ 
			\int_{\tau_i} h^2_{\tau_i}\,\big|\mu\,(1-\kappa)\,|\nabla u_h^n|^2\,v^n_h -\nu \big|^2\,\dx +\sum_{e_i\in \partial \tau_i\cap \mathcal{E}_h} \int_{e_i}\rho^2 \,h_{e_i}\,|\sjump{\nabla v_h^n}|^2\,\ds \label{esteta4.11bb}
		\end{align}
	where $c_{44}=\max\{c_{41},\, c_{42}\}$. 
\end{lemma}
\begin{proof}
	For every $\varphi \in \mathbb{V}_c$ and choosing an arbitrary $\varphi_h \in \mathbb{V}^n_{c,h}$, the equation \cref{3.18jprim} with \Cref{def4.1}, implying that 
	\begin{align}
		\big| \mathcal{J}^\prime(u^n_h; v^n_h)(\varphi)\big| &=~
		\Big|\sum_{\tau_i\in \mathscr{T}_h}\int_{\tau_i} [\mu\,(1-\kappa)\,|\nabla u_h^n|^2\,v^n_h -\nu]\,(\varphi-\varphi_h) \, \dx\nonumber\\
		&\hspace{0.5cm}+\sum_{ \tau_i\in \mathscr{T}_h}\int_{\tau_i} \rho\, \nabla v_h^n \cdot \nabla (\varphi - \varphi_h) \, \dx \Big| \nonumber\\ 
		&\leq \Big|\sum_{ \tau_i\in \mathscr{T}_h}\int_{\tau_i} [\mu\,(1-\kappa)\,|\nabla u_h^n|^2 \,v^n_h -\nu]\,(\varphi-\varphi_h)\,\dx\Big|\nonumber\\
		&\hspace{0.5cm}+\Big|\sum_{\tau_i\in \mathscr{T}_h}\int_{\partial \tau_i} \rho\, \nabla v_h^n \cdot \textbf{n}\; (\varphi - \varphi_h) \, \dx \Big| \nonumber\\ 
		& \leq \sum_{\tau_i\in \mathscr{T}_h} \int_{\tau_i}\Big|\mu\,(1-\kappa)\,|\nabla u_h^n|^2\,v^n_h -\nu\Big| \;|\varphi-\varphi_h|\,\dx\nonumber\\
		&\hspace{0.5cm}+\rho\;\sum_{ \tau_i\in \mathscr{T}_h} \sum_{ e_i\in \partial \tau_i\cap \mathcal{E}_h} \big|\sjump{\nabla v_h^n}|\;|\varphi - \varphi_h \big|\, \ds. \nonumber
	\end{align}
	Setting $\varphi_h=\pi_h^c \varphi$, then use of \Cref{4.10bbapr} of  \Cref{4.3lemmaapr} with $m=0$ and the Cauchy-Schwarz inequality leads to 
	\begin{align}
		&\big| \mathcal{J}^\prime(u^n_h; v^n_h)(\varphi)\big|\leq \max\{c_{41},\, c_{42}\} \times \Big\{\sum_{ \tau_i\in \mathscr{T}_h} \Big[\int_{\tau_i} h^2_{\tau_i}\,\Big|\mu\,(1-\kappa)\,|\nabla u_h^n|^2\,v^n_h -\nu\Big|^2\,\dx\nonumber\\
		&\hspace{0.5cm}+ \sum_{ e_i\in \partial \tau_i\cap \mathcal{E}_h} \int_{e_i} \rho^2 \,h_{e_i}\,|\sjump{\nabla v^n_h}|^2\,\ds \Big] \Big\}^{\frac{1}{2}} \Big\{\sum_{ \tau_i\in \mathscr{T}_h}\Big[ \|\nabla \varphi\|^2_{\omega_{\tau_i}}+ \sum_{ e_i\in \partial \tau_i \cap \mathcal{E}_h} \|\nabla \varphi\|^2_{\omega_{e_i}} \Big] \Big\}^{\frac{1}{2}}.\label{4.21BBbdd}
	\end{align}
	Setting $c_{44}=\max\{c_{41},\, c_{42}\}$, 
	this leads to the desired inequality \cref{4.12lemmabdd}. 
\end{proof}
Our next step is to establish the bound of $\mathcal{J}^\prime(u^n_h, v^n_h)$ in the dual norm of $ \mathbb{V}_c$, which is presented in the following lemma. 
\begin{lemma}
	Assume that all conditions of \Cref{4.5jbddlemma} hold. Then, 
	we have 
	\begin{align}
		\|\mathcal{J}^\prime(u_h^n; v_h^n)\|_{\overset{\ast}{\mathbb{V}_c}}~\leq ~ c_{44}\,\Big\{\sum_{ \tau_i\in \mathscr{T}_h^n}|\mathpzc{R}_{\tau}(u_h^n,v_h^n)|^2\Big\}^{\frac{1}{2}},
	\end{align} 
	where the constant $c_{44}$ is defined in  \Cref{4.5jbddlemma}.
\end{lemma}
\begin{proof}
	From the inequality \cref{4.12lemmabdd} and using the definition of dual norm to have
	\begin{align}
		\|\mathcal{J}^\prime(u^n_h; v^n_h)\|_{\overset{\ast}{\mathbb{V}_c}}=~\sup_{ \varphi \in \mathbb{V}_c} \frac{~|\mathcal{J}^\prime(u^n_h; v^n_h)(\varphi)|~}{ ~\|\varphi\|_{\mathbb{V}}}
		\leq ~ c_{44}\, \Big\{\sum_{ \tau_i\in \mathscr{T}_h^n}|\mathpzc{R}_{\tau}(u_h^n,v_h^n)|^2\Big\}^{\frac{1}{2}},
	\end{align} 
	where the indicators $\mathpzc{R}_{\tau}(u_h^n,v_h^n)$ is defined as in  \Cref{4.5jbddlemma}, this yields the desired result. 
\end{proof}
With assuming the decreasing tolerance $\Xi_{RF}$, the sequences generated via  adaptive \Cref{adapalg} will ultimately approach to a critical point of $\mathcal{J}(\cdot, \cdot)$, regardless of the initial conditions applied. The main result is stated in the following theorem.
\begin{theorem}[Main results] \label{Cngwithtol}
	Let $\Omega\subset \mathbb{R}^d$ be an open bounded domain. Further, we assume that there exists a sequence $\{(u_i^n,v_i^n)\}_{i=1}^{\infty}$ in $\mathbb{U}_g\times \mathbb{V}_c$ with $v^n_i(x)\in [0, 1]$ for $a.e.\; x\in \Omega$, and for some $\Upsilon^n_i$  with $\Upsilon^n_i \rightarrow 0$ as $i \rightarrow \infty$, such that 
	\begin{align}
	\mathcal{J}^\prime(u_i^n; v_i^n)(\varphi) \leq~ \Upsilon_i^n\; \| \varphi\|_{\mathbb{V}}, \quad  \; \forall n\in[1:\mathtt{N}_T]\label{4.30abinq} 
	\end{align}
	for all  $\varphi \in \mathbb{V}_c^\infty$. Again, we assume that the sequence $\{(u_i^n,v_i^n)\}_{i=1}^{\infty}$ is a bounded sequence in $\mathbb{U}\times \mathbb{V}$. Then, there exists a subsequence of $\{(u^n_i,v^n_i)\}_{i=1}^{\infty}$ (not-relabeled) and $(u, v)\in \widehat{\mathbb{U}}(0,T_f) \times \widehat{\mathbb{V}}(0,T_f)$ with $v(x,t)\in [0, 1]$  $a.e.\; x\in \Omega$, $t\in [0,T_f]$ such that $u_i^n$ and $v_i^n$ strongly converge to $u$ and $v$ as $i \rightarrow \infty$ in $\mathbb{U}$ and $\mathbb{V},$ respectively. Additionally, $u$ and $v$ satisfy 
	\begin{align}\label{4.6criticalpts}  
		\mathcal{J}^\prime(u; v)(\varphi)=~0  \quad \; \forall\,   \varphi \in \mathbb{V}_c^\infty. 
	\end{align}
	Hence, the function $\mathcal{J}(\ast;\cdot)$ has a critical point $v\in \mathbb{V}_c^\infty$, where $u \in \widehat{\mathbb{U}}(0,T_f)$ satisfies \cref{minu}--\cref{bdry}.
\end{theorem}
In the following lemma, we will state and prove the basic properties of sequences obtained using the adaptive \Cref{adapalg}. 
\begin{lemma}\label{lemma4.6bddseq}
	Let $\{(u_i^n,v_i^n)\}_{i=1}^\infty$ be a sequence generated via adaptive \Cref{adapalg} such that $\{(u^n_i,v^n_i)\}_{i=1}^\infty \subseteq \mathbb{U}^n_{g,h_i}\times \mathbb{V}^n_{c,h_i}$ at time level $t_n$, then the sequence $\{(u^n_i,v^n_i)\}_{i=1}^\infty$ holds the following properties
	\begin{enumerate}
		\item[\tt (i)]{\tt Pointwise boundedness of $v^n_i$:} $0 \leq v^n_i(x) \leq 1$ on $\Omega$ for all $i \in \mathbb{N}$,
		\item[\tt (ii)]{\tt Sequence boundedness:} the sequence $\{(u^n_i,v^n_i)\}_{i=1}^\infty$ is bounded in $\mathbb{U}\times \mathbb{V}$.
	\end{enumerate}   
\end{lemma}
\begin{proof}
	Property (i) is a direct consequence of Proposition \Cref{prop4.2:vh}.
	To prove property (ii), we show the boundedness of the sequences one by one in their respective spaces. \smallskip 
	
	\noindent
	\textbf{\tt Boundedness of $u^n_i$ in $\mathbb{U}$:}
	Substitute $u_h^n:=u_i^n$ and $v_h^n:=v_i^n$ into \cref{minuh}, and set $\psi_h= k_n\delta u_i^n$, we get
	\begin{align}
		&\varrho\, \big(\delta^2 u_i^n, k_n\delta u_i^n \big)+ \mathcal{B}^n_h(\ta_i^{n-1};u_i^n,  k_n\delta u_i^n )+\widehat{\mathcal{B}^n_h}(\ta_i^{n-1};u_i^n, k_n\delta u_i^n)=\big(f_i^n, k_n \delta u_i^n). \label{eq4.9bd}
	\end{align}
	It is easy to figure out the equivalent form of the first quantity on the left side of equation \cref{eq4.9bd}, as 
	\begin{align}
		&\varrho\, \big(\delta^2 u_i^n, k_n\delta u_i^n \big)=\frac{\varrho}{2}\big[\|\delta u^n_i\|^2-\|\delta u^{n-1}_i\|^2+k_n^2\|\delta^2 u^n_i\|^2\big].  \label{eq4.10bd}
	\end{align}
	Next, we find the equivalent form of  $\mathcal{B}^n_h(\ta_i^{n-1};u_i^n,  k_n\delta u_i^n )+\widehat{\mathcal{B}^n_h}(\ta_i^{n-1};u_i^n, k_n\delta u_i^n)$ of equation \cref{eq4.9bd}. In order to do this, we first evaluate the term
	\begin{align}
		\mathcal{B}^n_h(\ta_i^{n-1};u_i^n,  k_n\delta u_i^n )&=\big(\mu\,\ta_i^{n-1}\,\nabla u_i^n, \nabla  u_i^n \big)-\big(\mu\,\ta_i^{n-1}\,\nabla u_i^n, \nabla u_i^{n-1} \big)\nonumber\\
		&=\frac{\mu}{2}\big[\|(\ta_i^n)^{\frac{1}{2}}\nabla u_i^n\|^2-\|(\ta_i^{n-1})^{\frac{1}{2}}\nabla u_i^{n-1}\|^2+k_n^2\|(\ta_i^{n-1})^{\frac{1}{2}}\nabla \delta u_i^n\|^2\big]\nonumber\\
		&~~~- \frac{\mu}{2}\|(\ta_i^n-\ta_i^{n-1})^{\frac{1}{2}}\nabla u_i^n\|^2. \label{eq4.11bd}
	\end{align}
	Calculating the last term $-\frac{\mu}{2}\|(\ta_i^n-\ta_i^{n-1})^{\frac{1}{2}}\nabla u_i^n\|^2$ of the right-hand side of equation \cref{eq4.11bd} will yield 
	\begin{align}
		-\frac{\mu}{2}\|(\ta_i^n-\ta_i^{n-1})^{\frac{1}{2}}\nabla u_i^n\|^2&=-(1-\kappa)\,\mu\, k_n \int_{\Omega}v_i^n\, \delta v_i^n \, |\nabla u_i^n|^2\, \dx\nonumber\\
		&~~~+\frac{\mu}{2}\,(1-\kappa)\, k_n^2\, \|(\delta v_i^n )\,\nabla u_i^n\|^2. \label{eq4.12bd}
	\end{align}
	We now compute the first term of the right-hand side of equation \cref{eq4.12bd}. An implication of the fact ${(\mathcal{J}_h^n)}^{\prime}(u_i^n;v_i^n)(k_n\delta v_i^n)=0$ leads to
	\begin{align}
		&-  \mu\,(1-\kappa)\,k_n\,\int_{\Omega}v^n_i \delta v_i^n\, |\nabla u_i^n|^2 \, \dx=\int_{\Omega}\Big[\rho\,\nabla v_i^n \cdot \nabla (k_n\delta v_i^n) - \nu\,k_n\delta v_i^n \Big] \, \dx \nonumber\\
		&~~~=\frac{2\,\lambda_c}{c_w}\int_{\Omega}\Big[\frac{1}{2}\Big\{\epsilon\,|\nabla v_i^n |^2+ \frac{(1-v^n_i)}{\epsilon}\,\Big\}-\frac{1}{2}\Big\{\epsilon\,|\nabla v_i^{n-1} |^2+ \frac{(1-v^{n-1}_i)}{\epsilon}\,\Big\}\nonumber\\
		&~~~~~~~~~~~+\frac{\epsilon}{2}\Big\{|\nabla v_i^n |^2+|\nabla v_i^{n-1} |^2-2\,\nabla v_i^n\cdot \nabla v_i^{n-1}\Big\}\Big] \, \dx\nonumber\\
		&~~~=\frac{\lambda_c}{c_w}\,\big[ \mathcal{H}(v_i^n)-\mathcal{H}(v_i^{n-1})\big]+\frac{\lambda_c\,k_n^2\, \epsilon}{c_w}\,\|\nabla \delta v_i^n\|^2. \nonumber
	\end{align}
	Incorporating this in equation \cref{eq4.12bd}, we reach at 
	\begin{align}
		-\frac{\mu}{2}\|(\ta_i^n-\ta_i^{n-1})^{\frac{1}{2}}\nabla u_i^n\|^2&=\frac{\lambda_c}{c_w}\,\big[\mathcal{H}(v_i^n)-\mathcal{H}(v_i^{n-1})\big]+\frac{\lambda_c\,k_n^2\, \epsilon}{c_w}\,\|\nabla \delta v_i^n\|^2\nonumber\\
		&~~~+\frac{\mu}{2}\,(1-\kappa)\, k_n^2\, \|(\delta v_i^n )\,\nabla u_i^n\|^2. \label{eq4.13bd}  
	\end{align}
	Combining \cref{eq4.11bd} and \cref{eq4.13bd}, to have
	\begin{align}
		\mathcal{B}^n_h(\ta_i^{n-1};u_i^n,  k_n\delta u_i^n )&=\frac{\mu}{2}\big[\|(\ta_i^n)^{\frac{1}{2}}\nabla u_i^n\|^2-\|(\ta_i^{n-1})^{\frac{1}{2}}\nabla u_i^{n-1}\|^2+k_n^2\|(\ta_i^{n-1})^{\frac{1}{2}}\nabla \delta u_i^n\|^2\big]\nonumber\\
		&~~~+\frac{\lambda_c\,k_n^2\, \epsilon}{c_w}\,\|\nabla \delta v_i^n\|^2+ \frac{\lambda_c}{c_w}\,\big[\mathcal{H}(v_i^n)-\mathcal{H}(v_i^{n-1})\big]\nonumber\\
		&~~~+\frac{\mu}{2}\,(1-\kappa)\, k_n^2\, \|(\delta v_i^n )\,\nabla u_i^n\|^2. \label{eq4.14bd}
	\end{align}
	The next step is to look at the dissipation term, which is equivalent to
	\begin{align}
		\widehat{\mathcal{B}^n_h}(\ta_i^{n-1};u_i^n, k_n\delta u_i^n)= \eta\,k_n\|(\ta_i^{n-1})^{\frac{1}{2}}\, \nabla \delta u_i^n\|^2. \label{eq4.15bd}
	\end{align}
	Thus, bringing equations \cref{eq4.14bd} and \cref{eq4.15bd} together produces
	\begin{align}
		&\mathcal{B}^n_h(\ta_i^{n-1};u_i^n,  k_n\delta u_i^n )+\widehat{\mathcal{B}^n_h}(\ta_i^{n-1};u_i^n, k_n\delta u_i^n)=\frac{\mu}{2}\big[\|(\ta_i^n)^{\frac{1}{2}}\nabla u_i^n\|^2-\|(\ta_i^{n-1})^{\frac{1}{2}}\nabla u_i^{n-1}\|^2\nonumber\\
		&\hspace{2.5cm}+k_n^2\|(\ta_i^{n-1})^{\frac{1}{2}}\nabla \delta u_i^n\|^2\big]+ \frac{\lambda_c}{c_w}\,\big[ \mathcal{H}(v_i^n)- \mathcal{H}(v_i^{n-1})\big]+\frac{\lambda_c\,k_n^2\,\epsilon}{c_w} \|\nabla \delta v_i^n\|^2 \nonumber\\
		&\hspace{2.5cm}+\frac{\mu}{2}\,(1-\kappa)\, k_n^2\, \|(\delta v_i^n )\,\nabla u_i^n\|^2+ \eta\,k_n\|(\ta_i^{n-1})^{\frac{1}{2}}\, \nabla \delta u_i^n\|^2. \label{eq4.16bd}
	\end{align}
	Incorporating the values from equations \cref{eq4.10bd} and \cref{eq4.16bd} in the left-hand side of equation \cref{eq4.9bd} and then using the weighted Cauchy-Schwartz inequality leads to
	\begin{align}
		&\frac{\varrho}{2}\big[\|\delta u^n_i\|^2-\|\delta u^{n-1}_i\|^2+k_n^2\|\delta^2 u^n_i\|^2\big]+\frac{\mu}{2}\big[\|(\ta_i^n)^{\frac{1}{2}}\nabla u_i^n\|^2-\|(\ta_i^{n-1})^{\frac{1}{2}}\nabla u_i^{n-1}\|^2\nonumber\\
		&~~~~+k_n^2\|(\ta_i^{n-1})^{\frac{1}{2}}\nabla \delta u_i^n\|^2\big]+ \frac{\lambda_c}{c_w}\,\big[ \mathcal{H}(v_i^n)- \mathcal{H}(v_i^{n-1})\big]+\frac{\lambda_c\,k_n^2\,\epsilon}{c_w} \|\nabla \delta v_i^n\|^2\nonumber\\
		&~~~~+\frac{\mu}{2}\,(1-\kappa)\, k_n^2\, \|(\delta v_i^n )\,\nabla u_i^n\|^2+ \eta\,k_n\|(\ta_i^{n-1})^{\frac{1}{2}}\, \nabla \delta u_i^n\|^2 \nonumber\\ 
		&~~~\leq \frac{k_n}{2\,\eta\,\kappa}\,\|f^n_i\|^2_{-1} +\frac{k_n\,\eta\,\kappa}{2}\,\|\nabla \delta u_i^n\|^2.  \label{eq4.17bd}
	\end{align}
	Using the facts  $\kappa \leq \ta_i^{n-1} \leq 1$ and  $\|f^n_i\|_{-1} \leq \, c\,\|f^n_i\|^2$,  one may easily achieve 
	\begin{align}
		&\frac{\varrho}{2}\big[\|\delta u^n_i\|^2-\|\delta u^{n-1}_i\|^2+k_n^2\|\delta^2 u^n_i\|^2\big]+\frac{\mu}{2}\big[\|(\ta_i^n)^{\frac{1}{2}}\nabla u_i^n\|^2-\|(\ta_i^{n-1})^{\frac{1}{2}}\nabla u_i^{n-1}\|^2\nonumber\\
		&~~~~+k_n^2\|(\ta_i^{n-1})^{\frac{1}{2}}\nabla \delta u_i^n\|^2\big]+ \frac{\lambda_c}{c_w}\,\big[ \mathcal{H}(v_i^n)- \mathcal{H}(v_i^{n-1})\big]+\frac{\lambda_c\,k_n^2\,\epsilon}{c_w} \|\nabla \delta v_i^n\|^2\nonumber\\
		&~~~~+\frac{\mu}{2}\,(1-\kappa)\, k_n^2\, \|(\delta v_i^n )\,\nabla u_i^n\|^2+ \frac{k_n\,\eta\,\kappa}{2} \,\|\nabla \delta u_i^n\|^2 \leq \frac{k_n\, c^2}{2\,\eta\,\kappa}\,\|f^n_i\|^2.  
	\end{align}
	Summing over $1$ to $m$, $m\in [1:\tt{N}_T]$, we get
	\begin{align}
		&\big[\frac{\varrho}{2}\|\delta u^{m}_i\|^2+\frac{\mu}{2}\|(\ta_i^{m})^{\frac{1}{2}}\nabla u_i^{m}\|^2+ \frac{\lambda_c}{c_w}\, \mathcal{H}(v_i^{m})\big]+\sum_{n=1}^{m}k_n^2\Big[\frac{\varrho}{2}\,\|\delta^2 u^n_i\|^2  
		\nonumber\\
		&~+ \frac{\mu}{2}\,\|(\ta_i^{n-1})^{\frac{1}{2}}\nabla \delta u_i^n\|^2+\frac{\lambda_c\,\epsilon}{c_w}\,\|\nabla \delta v_i^n\|^2 +\frac{\mu\,(1-\kappa)}{2}\,\|(\delta v_i^n )\,\nabla u_i^n\|^2\Big]\nonumber\\
		&~+ \frac{\eta\,\kappa}{2}\,\sum_{n=1}^{m}k_n\,\|\nabla \delta u_i^n\|^2\leq \sum_{n=1}^{m}\frac{k_n\,c^2}{2\,\eta\,\kappa}\,\|f^n_i\|^2+\big[\frac{\varrho}{2}\|\delta u^{0}_i\|^2+\frac{\mu}{2} \|(\ta_i^0)^{\frac{1}{2}}\nabla u_i^{0}\|^2+\lambda_c\, \mathcal{H}(v_i^{0}) \big].  
	\end{align}
	Setting the discrete dissipation term with one summand over $k_n$ as
	\begin{align}
		\mathcal{D}_i^n&:=k_n\Big[\frac{\varrho}{2}\,\|\delta^2 u^n_i\|^2  
		+ \frac{\mu}{2}\,\|(\ta_i^{n-1})^{\frac{1}{2}}\nabla \delta u_i^n\|^2+\frac{\lambda_c\,\epsilon}{c_w}\,\|\nabla \delta v_i^n\|^2 \nonumber\\
		&~~~~~+\frac{\mu\,(1-\kappa)}{2}\,\|(\delta v_i^n )\,\nabla u_i^n\|^2\Big]+ \frac{\eta\,\kappa}{2}\,\|\nabla \delta u_i^n\|^2,\nonumber 
	\end{align}
	and hence
	\begin{align}
		&\frac{1}{2}\,\underset{n\in [1: \tt{N}_T]}{\max}\,\big[\varrho\,\|\delta u^{n}_i\|^2+\kappa\,\mu\,\|\nabla u_i^{n}\|^2+ 2\,\nu\,\|v_i^{n}\|^2+\rho\,\|\nabla v_i^{n}\|^2\big]+\sum_{n=1}^{\tt{N}_T}k_n\,\mathcal{D}_{i}^n\nonumber\\
		&~~~~~~~~~~~\leq c_{45} \times \Big[\sum_{n=1}^{\tt{N}_T}\|f^n_i\|^2+\|u_1\|^2+\|u_{0}\|^2_1 \Big],  \label{eq4.21bdgradu}
	\end{align}
since initially $(v_i^0)$ is chosen to be $1$, and the constant $ c_{45}=\max\{ \frac{\tilde{c}\,\,c^2}{2\,\eta\,\kappa}\,\frac{\mu}{2}, \,\frac{\varrho}{2}\}$, $\tilde{c}=\max_{n\in [1:~\tt{N_T}]}k_n$.
	This implies that the sequence $\{\|\nabla u^n_i\|\}_{i=1}^\infty$ is bounded. 
	
	Next, we show that the sequence $\{\|u^n_i\|\}_{i=1}^\infty$ is bounded in $\mathbb{U}$.  Then, 
	we apply Friedrichs' inequality and using \cref{eq4.21bdgradu} to obtain
	\begin{align}
		\|u^n_i\|&\leq~\|u^n_i-g\|+\|g\|\leq~c_F\;\|\nabla(u^n_i-g)\|+\|g\|\nonumber\\
		&\leq~c_F\; \|\nabla u^n_i\|+(c_F^2+1)^{1/2}\,\big(\|\nabla g\|^2+\|g\|^2\big)^{1/2}\nonumber\\
		&\leq~c_{46}\times \Big[\sum_{n=1}^{\tt{N}_T}\|f^n_i\|^2+\|u_1\|^2+\|u_{0}\|^2_1 +\|g\|^2_{\mathbb{V}}\Big]^{\frac{1}{2}}\nonumber
	\end{align}
with constant $c_{46}= \big(\frac{4\,c_F^2\;c_{45}^2}{\kappa^2\,\mu^2}+c_F^2+1 \big)^{1/2}$.	This demonstrates that the sequence $\{ u_i^n \}_{i=1}^\infty$ is bounded in $\mathbb{U}$. \\ \smallskip
	
	\noindent
	\textbf{\tt Boundedness of $v^n_i$ in $\mathbb{V}$:} For $t\in j_n, \, n\in[1:{\tt N_T}]$, to demonstrate the boundedness of the sequence $\{v^n_i\}_{i=1}^\infty$ in $\mathbb{V}$, we proceed by setting $\varphi_h=v_i^{n} \in \mathbb{V}^n_{c,h_i}$ in $(\mathcal{J}^n_{h_i})^\prime(u_i^n;v_i^n)(\varphi_h)=~0$, to obtain the following.
	\begin{align}
		& \int_{\Omega}\rho\, |\nabla v^{n}_{i} |^2\, \dx+ \underaccent{ {\bf \geq 0}}{\underbrace{\int_{\Omega}\mu\,(1-\kappa)\,|\nabla u^n_{i}|^2\;(v^{n}_{i})^2\, \dx}}~=~\int_{\Omega}\nu \,v^{n}_{i}\, \dx. \nonumber 
	\end{align}
	Exploiting property (i) of Lemma \ref{lemma4.6bddseq}, this leads to 
	\begin{align}
		\|\nabla v^{n}_{i} \|\,\leq~ c_{47}
	\end{align}
	with $c_{47}= \frac{\nu}{\rho}\times \meas(\Omega)$. The boundedness of the sequence $\{\|\nabla v^n_i\|\}_{i=1}^\infty$ and the pointwise boundedness of $\{ v^n_i \}_{i=1}^\infty$ in the $\mbox{L}^2$-norm implies that the sequence $\{ v^n_i \}_{i=1}^\infty$ is bounded in $\mathbb{V}$. We have thus established that the sequence $\{(u^n_i, v^n_i)\}_{i=1}^\infty$ is bounded in $\mathbb{U}\times \mathbb{V}$, which concludes the proof of the lemma.
\end{proof}
We are in the position to establish the proof main results of \Cref{Cngwithtol} with the help of  \Cref{ugdbound}.
\begin{proof}[Proof of \Cref{Cngwithtol}]  The proof is divided into two parts. In the first part, we show that the existence sequences converge to $u$ and $v$, while the second step shows that $v$ admits as a critical point of $\mathcal{J}(\ast;\cdot)$, with extra steps added as needed to make things clearer. \smallskip
	
\noindent
\textbf{\tt Step-1:} We first establish the existence of a convergent subsequence of $\{(u_i,v_i)\}_{i=1}^{\infty}$ in $\mathbb{U}_g\times \mathbb{V}_c$. By  \Cref{lemma4.6bddseq}, the sequence $\{(u^n_i,v^n_i)\}_{i=1}^{\infty}$ is bounded in $\mathbb{U}\times \mathbb{V}$. Since $\mathbb{U}$ is a Hilbert space, boundedness implies the existence of a weakly convergent subsequence. Specifically, there exists a subsequence (not relabeled) such that $(u^n_i,v^n_i) \overset{w}{\longrightarrow} (u,v)$ as $i\rightarrow \infty$ in $\mathbb{U}\times \mathbb{V}$. In particular, since $\mathbb{U}_g$ is a closed and convex subset of $\mathbb{U}$, it is also weakly closed. Consequently, the weak limit $u$ is an element of $\mathbb{U}$.  
	
	Let us define a subset $\widehat{\mathbb{V}}$ of $\mathbb{V}$ by  $$\widehat{\mathbb{V}}:=\{\hat{v} \in \mathbb{V}_c :~  0\leq \hat{v}(x)\leq 1~~ a.e. ~~ x \in \Omega\}.$$
	One can easily see that $\widehat{\mathbb{V}}$ is a closed and convex subset of $\mathbb{V}$. Since $v^n_i\in \widehat{\mathbb{V}},\, \forall\, i\in \mathcal{N}$, hence, $0\leq v(x)\leq 1,\; a.e., \, x\in \Omega$. However, the compact embedding $\mbox{H}^1(\Omega) \hookrightarrow_c \mbox{L}^2(\Omega)$ ensures that the sequence $\{(u_i^n,v_i^n)\}_{i=1}^{\infty}$ converges strongly to $(u,v)$ in $\mbox{L}^2(\Omega)\times \mbox{L}^2(\Omega)$.
	
	The following step shows that the phase field variable $v$ satisfies $\mathcal{J}(u; v)(\varphi)=~0,$ $   \forall \,  \varphi \in \mathbb{V}_c^\infty$, where $u$ is the solution of the problem \cref{minu}--\cref{bdry}. \smallskip
	
	\noindent 
	\textbf{\tt Step-2:} Considering equation \cref{3.18jprim}, we figure out
	\begin{align}
		\mathcal{J}^\prime(u; v)(\varphi)&=~\int_{\Omega}\big[\rho\, \nabla v \cdot \nabla \varphi -\nu\, \varphi +\mu\,(1-\kappa)\,v\, \varphi\,|\nabla u|^2\big] \, \dx\nonumber\\
		&=~\int_{\Omega}\Big[\rho\, \nabla (v-v_i^n) \cdot \nabla \varphi+ \mu\,(1-\kappa)\,|\nabla u|^2 \,(v-v_i^n)\,\varphi \Big]\,\dx\nonumber\\
		&~~~~+ \int_{\Omega}\Big[\rho\, \nabla v_i^n \cdot \nabla \varphi - \nu \, \varphi + \mu\,(1-\kappa)\,|\nabla u_i^n|^2\,\,v_i^n\,\varphi \Big]\,\dx\nonumber\\
		&~~~~+\int_{\Omega} \mu\,(1-\kappa)\,\big(|\nabla u|^2-|\nabla u_i^n|^2\big)\,v_i^n\,\varphi \, \dx\nonumber\\
		&:=~\mathscr{P}_i^n+\mathscr{Q}_i^n+\mathscr{R}_i^n. \label{barxyzbb} 
	\end{align}
	In order to complete the proof, we need to show that $\mathscr{P}_i^n,\, \mathscr{Q}_i^n,$ and $\mathscr{R}_i^n$ vanish as $i$ tends to infinity. Each of these convergences will be examined independently.\smallskip
	
	\noindent
	\textbf{(i)} 
	We begin by analyzing the first term $\mathscr{P}_i^n$. We consider
	\begin{align}
		|\mathscr{P}_i^n|&= \Big|\int_{\Omega}\Big[\rho\, \nabla (v-v_i^n) \cdot \nabla \varphi+ \mu\,(1-\kappa)\,|\nabla u|^2\,(v-v_i^n)\,\varphi \Big]\,\dx\Big| \nonumber\\
		&\leq  \rho\, \Big|\int_{\Omega} \nabla (v-v_i^n) \cdot \nabla \varphi\, \dx \Big|+ \Big| \int_{\Omega} \mu\,(1-\kappa)\,|\nabla u|^2\, (v-v_i^n)\,\varphi \, \dx \Big| \nonumber\\
		&\leq  \rho\, \Big|\int_{\Omega} \nabla (v-v_i^n) \cdot \nabla \varphi\, \dx \Big|+ \mu\,|1-\kappa|\,\|\varphi\|_{L^\infty(\Omega)}\,\Big| \int_{\Omega} |\nabla u|^2\, (v-v_i^n)\, \dx \Big|. \label{4.20tildexi}
	\end{align}
	Since, $\nabla v_i^n \overset{w}{\longrightarrow} \nabla v$  in $(L^2(\Omega))^d$ and $v_i^n \longrightarrow  v$ in $L^2(\Omega)$. 
	It is sufficient to demonstrate that as $i$ goes to infinity, the expression $\int_{\Omega} |v-v_i^n|\; |\nabla u|^2\,\dx$ vanishes.
	
	To proceed, we assume that the sequence $\{v_i^n \}_{i=1}^\infty$ has a convergent subsequence $\{ v^n_{i_k} \}_{k=1}^\infty$, such that $ v^n_{i_k} \longrightarrow v$ {\it a.e.} in $\Omega$. Additionally, we suppose that this subsequence fulfills, for fixed $\psi\in \mathbb{U}$,
	$$\lim_{k \rightarrow \infty} \int_{\Omega} |v-v^n_{i_k}|\;|\nabla \psi|^2\,\dx=\limsup_{i \rightarrow \infty} \int_{\Omega} |v-v_i^n|\;|\nabla \psi|^2\,\dx.$$
	Applying the Dominated Convergence theorem, we arrive at
	$$\lim_{k \rightarrow \infty} \int_{\Omega} |v- v^n_{i_k}|\;|\nabla \psi|^2\,\dx=0,$$
	and, hence
	$$\limsup_{i \rightarrow \infty} \int_{\Omega} |v-v^n_i|\;|\nabla \psi|^2\,\dx=0.$$
	Thus we  have
	\begin{align}
		\int_{\Omega} |v-v^n_i|\;|\nabla \psi|^2\,\dx=0,\quad \text{i} \rightarrow \infty. \label{4.27dc}
	\end{align}
	Using the fact from equation \cref{4.27dc} with $\psi=u$, equation \cref{4.20tildexi} concludes that $|\mathscr{P}^n_i|\longrightarrow 0$ as $i$ tends to $\infty$, and therefore $\mathscr{P}^n_i$ vanished.\smallskip
	
	Next, we show the convergence of the sequence $\mathscr{Q}^n_i$. \smallskip
	
	\noindent
	\textbf{(ii)} 
	It is clear that the expression \cref{4.30abinq} can control the sequence $\mathscr{Q}^n_i$. Specifically,
	$$|\mathscr{Q}^n_i|=\big|\mathcal{J}'(u_i^n;v_i^n)(\varphi)\big|\leq~ \Upsilon^n_i\,\|\varphi\|_{\mathbb{V}}.$$ 
	As $i \rightarrow \infty$, we conclude that $|\mathscr{Q}^n_i| \longrightarrow 0$, since  $\Upsilon^n_i \longrightarrow 0$. Thus, we obtain $\mathscr{Q}^n_i\longrightarrow 0$. \smallskip
	
	Finally, we concentrate on proving the convergence of $\mathscr{R}^n_i$.  \smallskip
	
	\noindent
	\textbf{(iii) \tt Convergence of sequence $\mathscr{R}^n_i$.} Consider 
	\begin{align}
		|\mathscr{R}^n_i|&=\Big|\int_{\Omega} \mu\,(1-\kappa)\,\big(|\nabla u|^2-|\nabla u_i^n|^2\big)\,v_i^n\,\varphi \, \dx\Big| \nonumber \\
		& \leq \mu\,(1-\kappa)\, \underset{x\in \Omega}{\sup}\,|\varphi|\; \|\nabla u+ \nabla u_i^n\|\,\|\nabla u- \nabla u_i^n\|. \label{4.28gu}
	\end{align}
	Since $0\leq v_i^n\leq 1$. Note that $\{ \|\nabla u+ \nabla u_i^n\| \}_{i=1}^\infty$ is a bounded sequence; then, invoking \Cref{ugdbound} with the fact that $v_i^n$ converges to $v$ as $i\to \infty$, one can easily obtain that $|\mathscr{R}^n_i| \to \infty$. 
	
	Bringing all the convergence results together $\mathscr{P}^n_i,\, \mathscr{Q}^n_i,$ and $\mathscr{R}^n_i$, we conclude that $\mathcal{J}'(u; v)(\varphi)=0$ for all $\varphi\in \mathbb{V}^\infty_c$. \smallskip
	
	To complete our analysis, we must show the strong convergence of the sequence $\{\nabla v_i^n\}_{i=1}^\infty$ in $(L^2(\Omega))^d$, which will be dealt with next.\smallskip
	
	\noindent 
	\textbf{\tt (a)} To prove the strong convergence of $\{\nabla v_i^n\}_{i=1}^\infty$, we consider
	\begin{align}
		&\rho\, \|\nabla v_i^n-\nabla v\|^2 \leq~\int_{\Omega}  \big(\mu\, (1-\kappa)\,|\nabla u|^2+\nu\, \big)\,|v_i^n-v|^2\,\dx \nonumber\\
		&~~~~~~~+\int_{\Omega} \rho\, (\nabla v_i^n-\nabla v)\cdot  (\nabla v_i^n-\nabla v)\,\dx \nonumber\\
		&~~~~~~= \int_{\Omega} \big[\mu\,(1-\kappa)\,|\nabla u|^2\,v_i^n\,(v_i^n-v)+\nu\,v_i^n\,(v_i^n-v)+ \rho\,\nabla v_i^n\cdot  (\nabla v_i^n-\nabla v)\big]\,\dx \nonumber\\ 
		&~~~~~~~~-\int_{\Omega}\big[\mu\,(1-\kappa)\,|\nabla u|^2\,v\,(v_i^n-v) +\nu\,v\,(v_i^n-v)+ \rho\,\nabla v \cdot  (\nabla v_i^n-\nabla v)\big]\,\dx \nonumber\\
		&~~~~~~= \mathcal{J}(u_i^n;v_i^n)(v_i^n-v)-\mathcal{J}(u;v)(v_i^n-v)\nonumber\\
		&~~~~~~~~+\int_{\Omega}  \mu\,(1-\kappa)\,\big(|\nabla u|^2-|\nabla u_i^n|^2\big)\,v_i^n\,(v_i^n-v)\, \dx. \nonumber
 	\end{align}
	Utilizing the fact that  $\mathcal{J}^\prime(u;v)(v_i^n-v)=0$, for $\varphi \in \mathbb{V}_c^\infty$, in conjunction with \cref{4.30abinq}  and  \cref{4.37gradu}, we derive
	\begin{align}
		\rho\, \|\nabla v_i^n-\nabla v\|^2& \leq~ 
		\big|\mathcal{J}^\prime(u_i^n;v_i^n)(v_i^n-v)\big|+2\,\mu\,|1-\kappa\,\|\nabla u- \nabla u_i^n\| \,  \|\nabla u+\nabla u_i^n\| \nonumber\\
		&\leq~  \Upsilon_i\,\|\nabla(v_i^n-v)\|+ 2\, \mu\,\big|1-\kappa \big|\,\|\nabla u+\nabla u_i^n\| \,  \|\nabla u- \nabla u_i^n\|,
		\label{4.37gradu}
	\end{align}
	since $0\leq v,~ v_i^n\leq 1$ and the term $\{\|\nabla u+\nabla u_i^n\|_{i=1}^{\infty}$ is bounded, hence
	\begin{align}
		\|\nabla v_i^n-\nabla v\|&\leq \frac{1}{\rho}\, \Big[\Upsilon_i+2\, \mu\,\big|1-\kappa\big|\, \|\nabla u+\nabla u_i^n\|\,  \|\nabla u- \nabla u_i^n\| \Big]. \label{4.29999bd}
	\end{align}
	Note that right-hand side of equation \cref{4.29999bd} vanishes as  $\nabla u_i^n\longrightarrow \nabla u$ and  $\Upsilon_i \longrightarrow 0$ as $i \rightarrow \infty$. Thus, we conclude that $\{\nabla v_i^n\}_{i=1}^\infty$ converges strongly to $\nabla v$ in $(L^2(\Omega))^d$. Hence, the proof is completed.
\end{proof}
\section{Computational results}\label{expres}
This section numerically validates our theory and demonstrates the efficacy of the proposed adaptive algorithm. 
Simulations were performed using a custom C++ framework built on the \textsf{deal.II} library \cite{arndt2021deal}, extending our previous work \cite{fernando2025b,manohar2025adaptive,manohar2025convergence}.
Key parameters include: regularization $\kappa = 10^{-10}$, irreversibility tolerance $10^{-2}$, critical energy release rate $\lambda_c=1.0$, and wave speed $c_w=8/3$.  The length scale is coupled to the local mesh size via $\varepsilon=5h_f$, which defines the density $\rho=10\sqrt{h_f}$ and viscosity $\nu = (10\sqrt{h_f})^{-1}$. At each time step, a D\"{o}rfler marking strategy refines $20\%$ of cells with the highest error and coarsens $5\%$ with the lowest.  We use an iterative staggered scheme where the inner displacement and phase-field solve converge when the $L^{\infty}$-norm of successive iterates is below $10^{-10}$.  The simulation runs for 1600 time steps, allowing the wave to traverse the domain. 
We analyze the final phase-field state ($v$) and the corresponding system energies.
\begin{exam} \label{exm1}
	We consider a rectangular domain $\Omega=[0,3]\times [0,3]$ with a slit $[0, 1.5] \times \{1.5\}$ which is shown in  \Cref{crackdomain}.  	The incremental anti-plane displacement function $g(x,t)$ with $x=(x_1,x_2)$, is given by 
	\begin{align*}
		g(x,t)=g_0(t)\,\mathrm{sgn}(x_2), \quad \text{where} \quad g_0(t)=
		\begin{cases}
			\dfrac{\varepsilon_v\,t^2}{2\,t_s}   \quad  \text{if} \quad  t \in [0, t_s], \\
			\varepsilon_v\,t -\dfrac{\varepsilon_v\,t_s}{2}  \quad  ~\text{if} \quad t \in [t_s, T_g], 
		\end{cases}
	\end{align*}
	where $\varepsilon_v=0.9$, while $t_s$ and $T_g$ are arbitrary chosen threshold function such that $t_s \ll T_g$.
	\begin{figure}[htp] 
		\centering
		\begin{tikzpicture}[scale=1.1]
			\tikzset{myptr/.style={decoration={markings,mark=at position 1 with {\arrow[scale=3,>=stealth]{>}}},postaction={decorate}}}
			\filldraw[draw=black, thick] (0,0) -- (6,0) -- (6,4) -- (0,4) -- (0,0);
			\shade[inner color=gray!10, outer color=cyan!40] (0,0) -- (6,0) -- (6,4) -- (0,4) -- (0,0);
			\node at (-0.35,-0.35)  {$(0,0)$};
			\node at (6.24,-0.35)   {$(3,0)$};
			\node at (-0.35,4.35)   {$(0,3)$};
			\node at (6.25, 4.35)   {$(3,3)$};
			\node at (6.4, -0.1)[anchor=west, rotate=90]{$\Gamma_2:~ \mathbf{n}\cdot \nabla u=0 ~\&~ \mathbf{n}\cdot \nabla v=0$};
			\node at (3, -0.35) {$\Gamma_{1}:~ \mathbf{n}\cdot \nabla u=0  ~\&~ \mathbf{n}\cdot \nabla v=0$};
			\node at (-0.65, 3.8)[anchor=east, rotate=90]{$u=-g_0(t) ~~~~~~ u=g_0(t)$};
			\node at (-1.1, 3.5)[anchor=east, rotate=90]{$ \Gamma_{4}~ \& ~\Gamma_5:~\mathbf{n}\cdot \nabla v=0 $};
			\node at (3, 4.35) {$\Gamma_{3}:~\mathbf{n}\cdot \nabla u=0 ~~\& ~~ \mathbf{n}\cdot \nabla v=0 $};
			\draw [myptr](-0.4, 1.8)--(-0.4, 0);
			\draw [myptr](-0.4, 2.2)--(-0.4, 4);
			\node at (-0.2, 3.4)[anchor=east, rotate=90]{$ \Gamma_{4}$};
			\node at (-0.2, 1.4)[anchor=east, rotate=90]{$ \Gamma_{5}$};
			\draw [line width=0.5mm, black!80]  (0,2.0) -- (2,2.0);
			\node at (2.6,1.95)[anchor=east]{$\Gamma_c$};
			\node at (1.85,1.7)[anchor=east]{$\mathbf{n}\cdot \nabla v =0$};
			\node at (1.85,2.3)[anchor=east]{$\mathbf{n}\cdot \nabla u =0$};
			\def\xOrigin{7.5}
			\def\yOrigin{0.5} 
			\draw[->, color=red, thick] (\xOrigin, \yOrigin) -- (\xOrigin + 1, \yOrigin) node[right] {$x$};
			\draw[->, color=blue, thick] (\xOrigin, \yOrigin) -- (\xOrigin, \yOrigin + 1) node[above] {$y$};
		\end{tikzpicture}
		\caption{A domain and the boundary indicators.}
		\label{crackdomain}
	\end{figure}
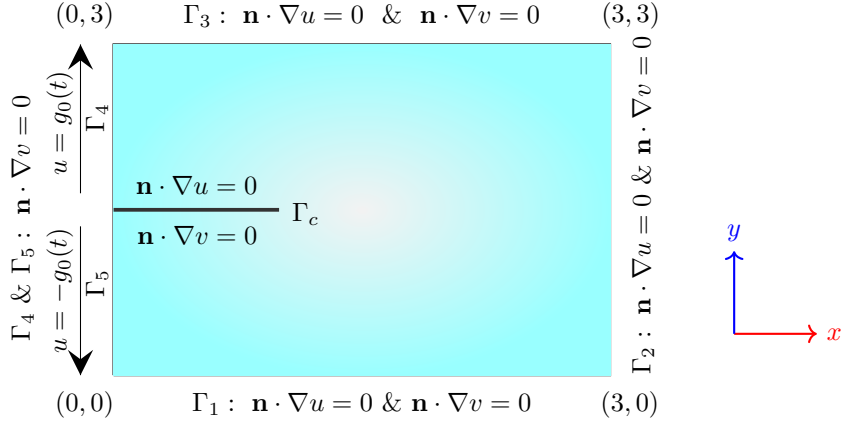
\end{exam}
\begin{figure}
	\centering
	\begin{subfigure}{0.45\textwidth}
		\includegraphics[width=\linewidth]{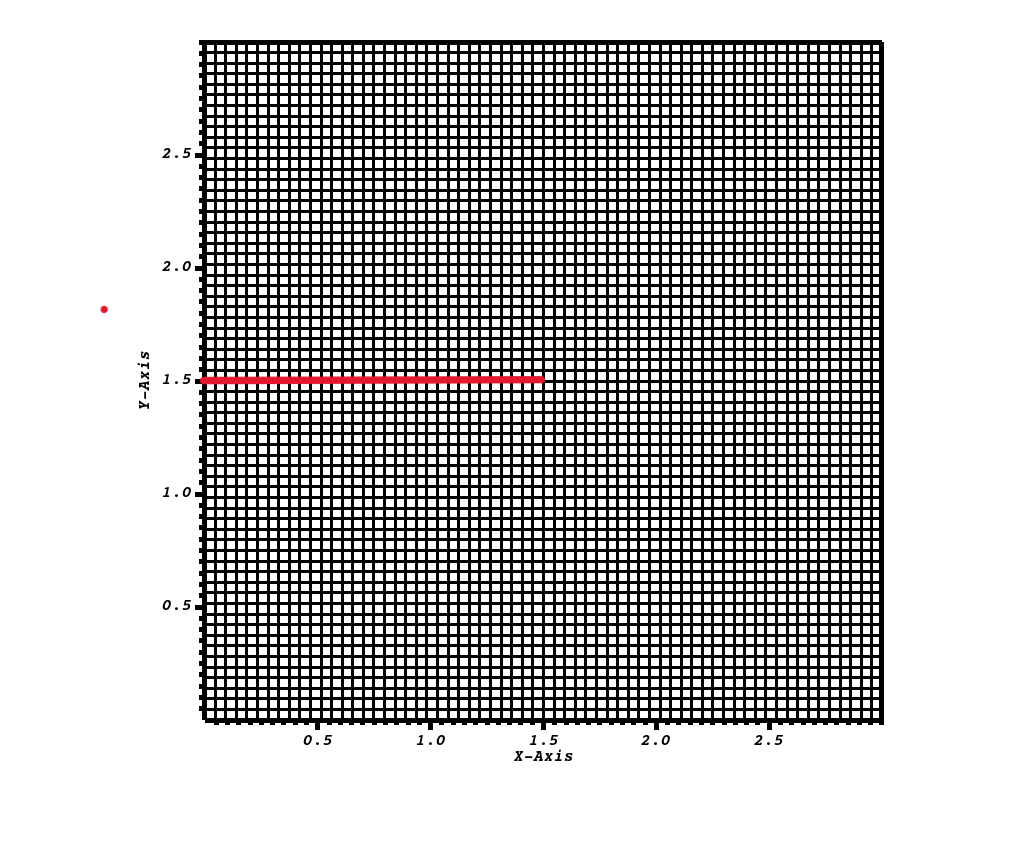}
		\caption{Coarse mesh}
		\label{fig:sub1ini}
	\end{subfigure}\hfill
	\begin{subfigure}{0.45\textwidth}
		\includegraphics[width=\linewidth]{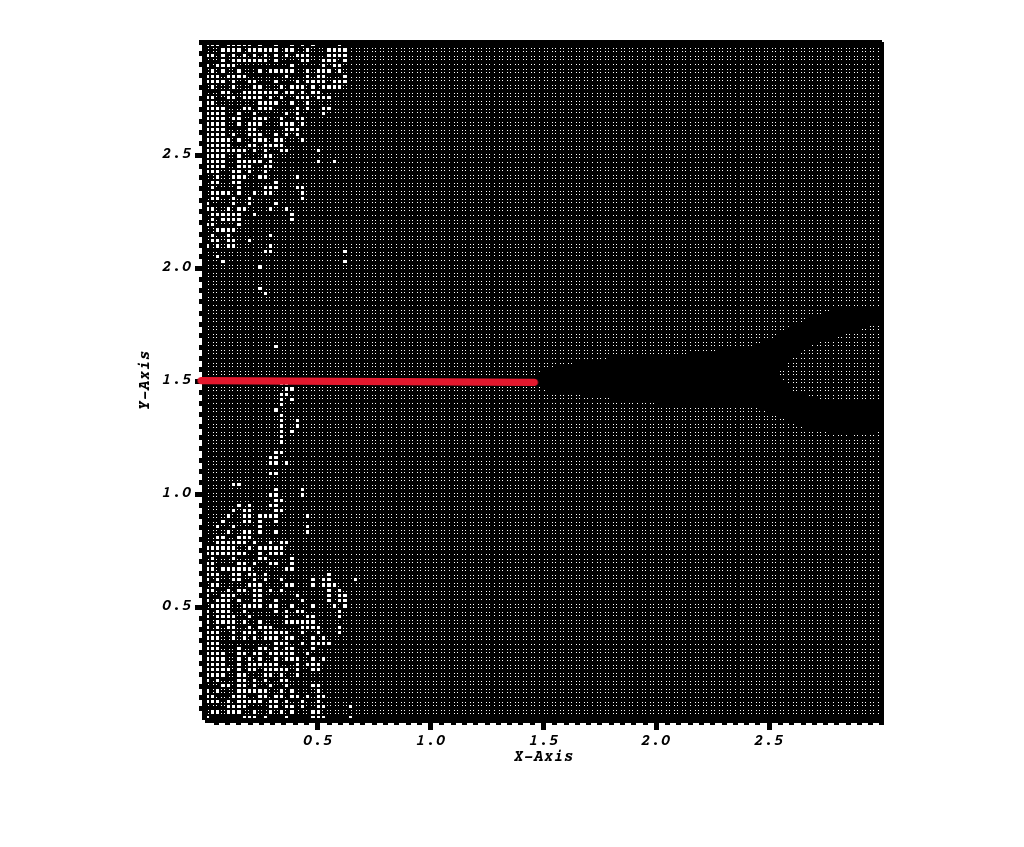}
		\caption{Final mesh at the last time-step}
		\label{fig:sub2final}
	\end{subfigure}
	\caption{Illustration of the computational mesh, with the initial slit highlighted by the red line.}
	\label{fig:meshes}
\end{figure}
\Cref{fig:meshes} illustrates the state of the adaptively refined computational mesh at the initial  and final stages (cf., \Cref{fig:sub1ini} and \Cref{fig:sub2final}, respectively) of the simulation. The simulation commences on a uniform base mesh composed of $4,096$ quadrilateral cells ($64 \times 64$ grid), which corresponds to $4,257$ degrees of freedom (DOFs) when discretized with bilinear $\mathbb{Q}_1$ finite elements. The adaptive mesh refinement is driven by a marking threshold; any cell with an error indicator exceeding a tolerance of $10^{-3}$ is flagged for refinement. This strategy results in immediate local refinement around the initial slit, increasing the mesh complexity to $6,187$ active cells and $6,428$ DOFs after the first time step alone.

By the final time step ($t=1600$), the mesh has been extensively adapted to resolve the propagating fracture, comprising a total of $85171$ active cells. Over the course of the simulation, more than half of these final cells have been subject to at least one refinement event, highlighting the dynamic nature of the mesh and the efficiency of the AMR strategy. To maintain a well-conditioned system, a maximum of four refinement levels is permitted for any cell descending from the base mesh. This constraint imposes a minimum mesh size of $h_{\min} \approx 0.002929$ in the most refined regions near the crack.

\begin{figure}
	\centering
	\begin{subfigure}{0.45\textwidth}
		\centering
		\includegraphics[width=\linewidth]{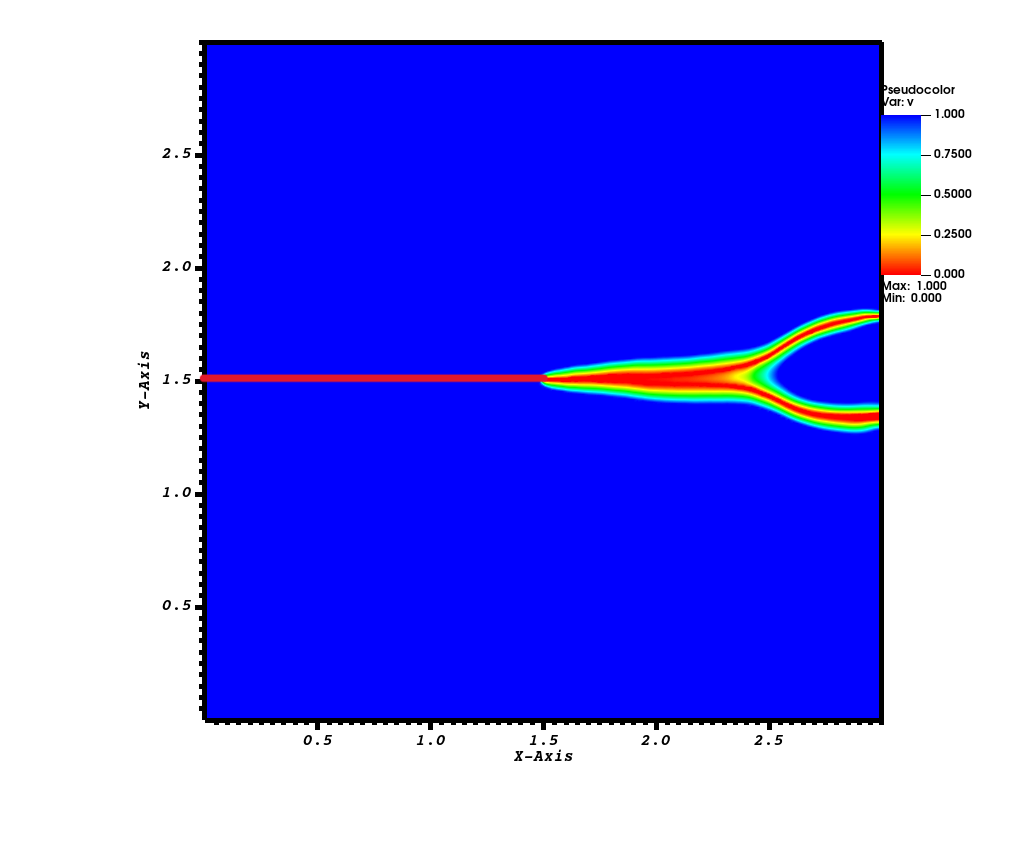}
		\caption{Final phase-field ($v$).}
		\label{fig:sub-a}
	\end{subfigure}
	\hfill
	\begin{subfigure}{0.45\textwidth}
		\centering
		\includegraphics[width=\linewidth]{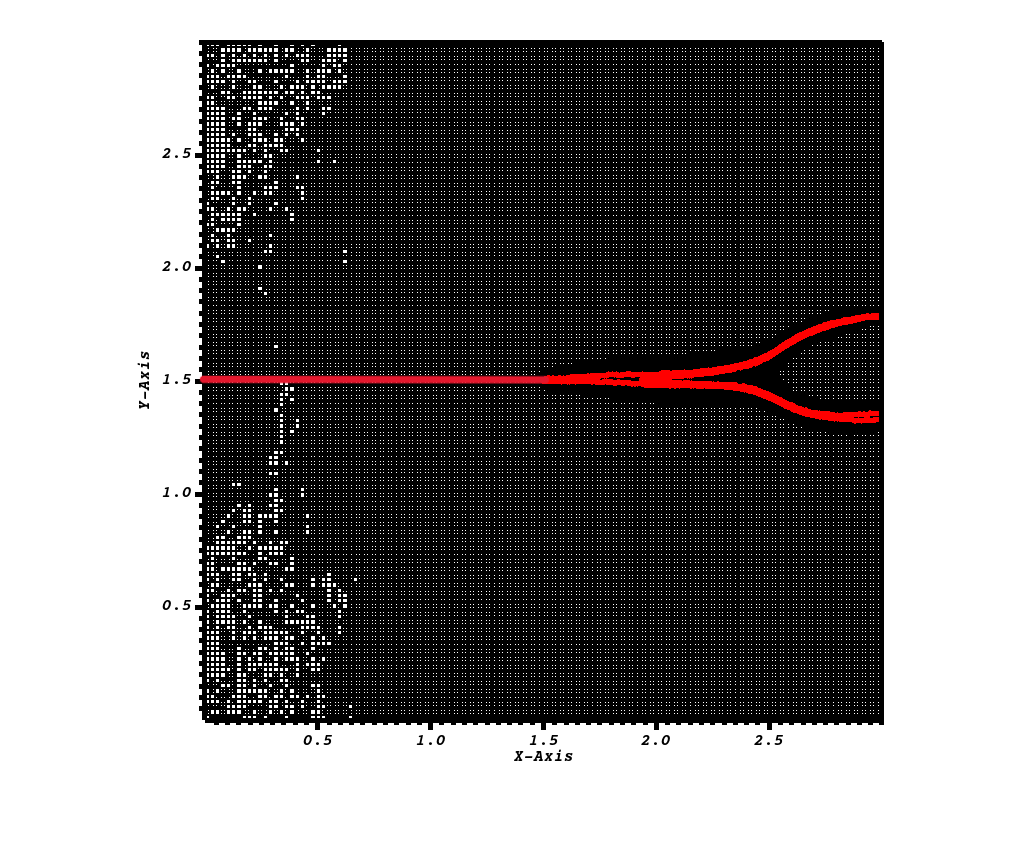}
		\caption{Adapted mesh with $v$-contour.}
		\label{fig:sub-b}
	\end{subfigure}
	\caption{Computational results at the final time step. 
		\textbf{Left:} Final phase-field ($v$), indicating the fracture path. 
		\textbf{Right:} Adapted mesh showing the contour for $0 \leq v \leq 0.01$.}
	\label{fig:v_mesh_stress}
\end{figure}
\begin{figure}
	\centering 
	\begin{subfigure}{0.49\textwidth}
		\centering
		\includegraphics[width=\linewidth]{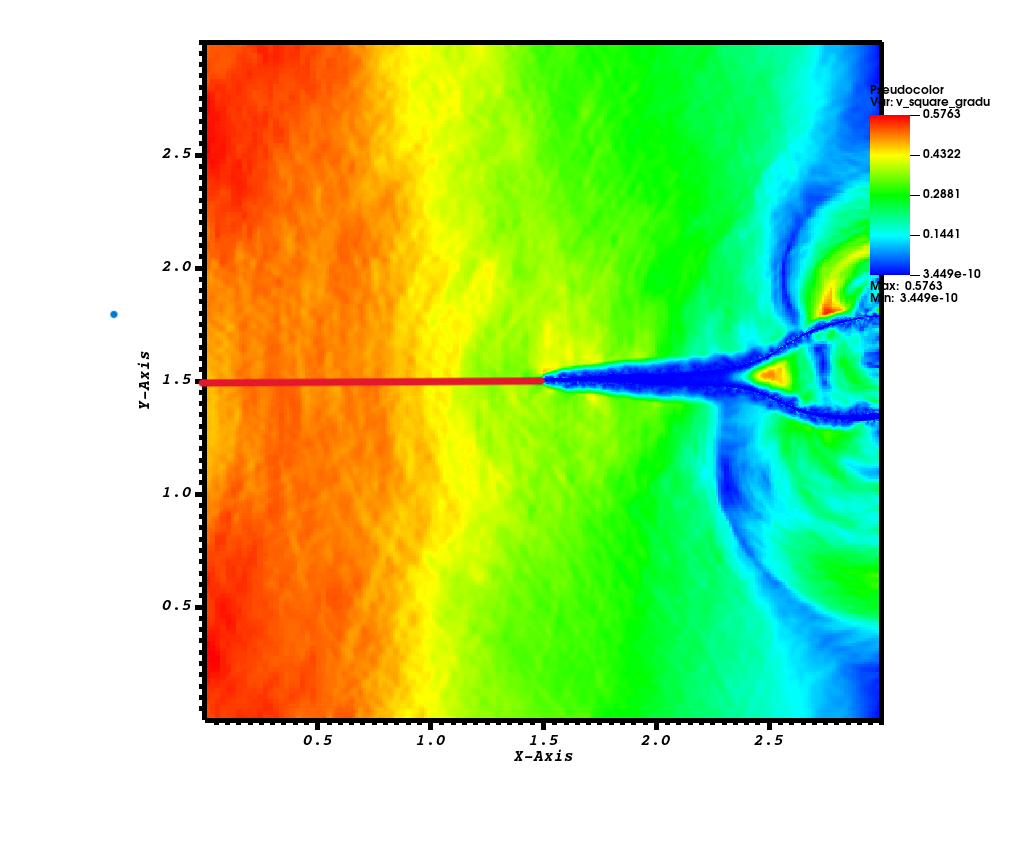}
		\caption{Regularized stress field.}
		\label{fig:sub-c}
	\end{subfigure}
	\hfill
	\begin{subfigure}{0.49\textwidth}
		\centering
		\includegraphics[width=\linewidth]{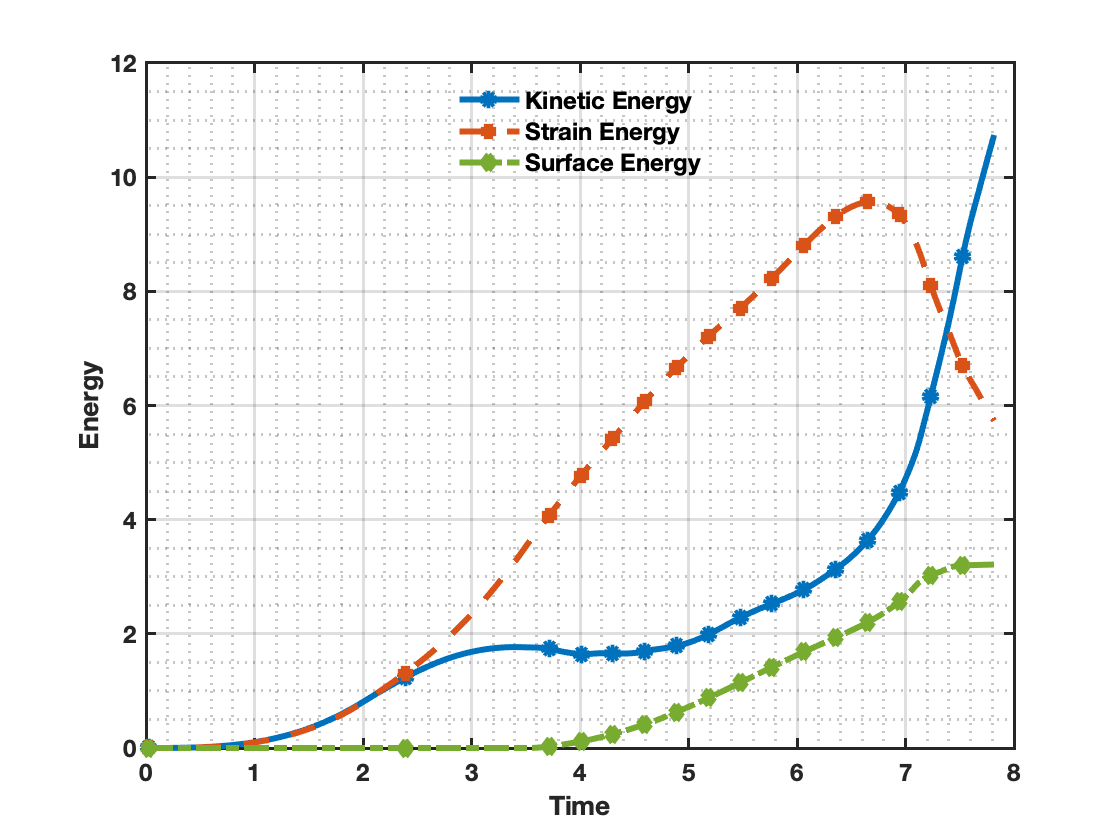}
		\caption{Evolution of the energies as a function of time.}
		\label{fig:energies} 
	\end{subfigure}
	\caption{\textbf{Left:} Regularized stress field. 
		\textbf{Right:} Evolution of kinetic, strain, and surface energies as a function of time.}
	\label{fig:energies123} 
\end{figure}
\Cref{fig:sub-a} presents the computational results for the final state of the phase-field variable ($v$), which provides a regularized approximation of the damage field. The simulation successfully captures a complex fracture pattern, including a clear depiction of crack branching. These results underscore the efficacy of the error indicator proposed in this work. The mesh adaptation is not only concentrated precisely along the entire damage zone (i.e., where $0 < v < 1$) but is also judiciously distributed into the surrounding regions where the solution field, while smoother, exhibits significant gradients. This demonstrates the indicator's crucial ability to resolve both singular and regular features of the solution. To visualize the extent of the diffuse damage, the right panel \Cref{fig:sub-b}  displays a contour plot for $v$ in the range $0 \le v \le 0.01$. The upper bound for this contour is intentionally chosen to match the irreversibility threshold, thereby highlighting the region where the material has undergone incipient damage.
Further,   \cref{fig:sub-a}  visualizes the regularized elastic energy density field, given by the expression $((1-\kappa)v^2 + \kappa) \| \nabla \mathbf{u} \|^2$. This scalar quantity serves as a proxy for the local stress state, with the color map selected such that the darkest blue hues indicate regions of intense stress concentration, which are predominantly localized ahead of the propagating crack tips (cf., \Cref{fig:sub-b} ).

\Cref{fig:energies} plots the temporal evolution of the system's primary energy components: kinetic, elastic strain, and surface energies. In the initial phase, the domain accumulates elastic strain energy under the applied load. The plot reveals a critical transition at approximately $t \approx 3.8$, which marks the onset of fracture initiation. At this juncture, the stored strain energy reaches its peak and begins to decrease sharply. This release of elastic energy is converted into surface energy, as evidenced by the concurrent and rapid increase in the surface energy curve. This behavior is consistent with Griffith's energy criterion and signifies that the stored potential energy is being dissipated to drive dynamic crack propagation from the tip of the initial slit. 
\begin{figure}
	\centering 
	\begin{subfigure}{0.49\textwidth}
		\centering
		\includegraphics[width=\linewidth]{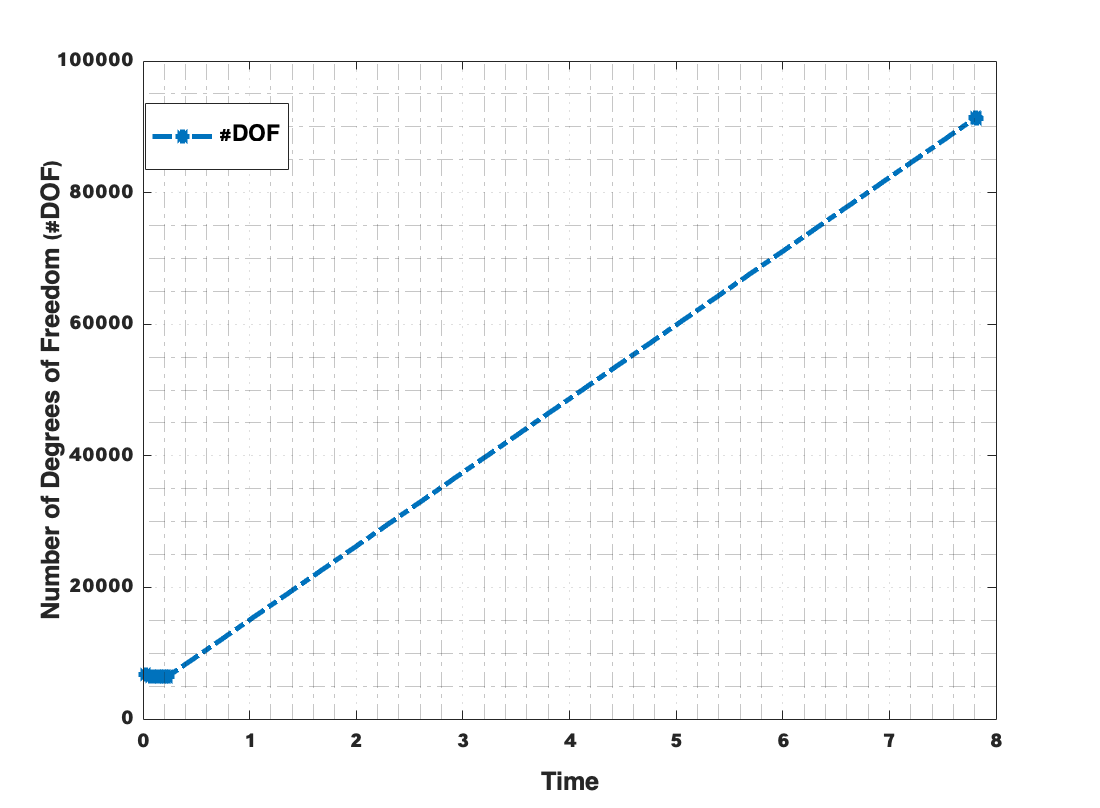}
		\caption{\# DOF vs Time.}
		\label{doft}
	\end{subfigure}
	\hfill
	\begin{subfigure}{0.49\textwidth}
		\centering
		\includegraphics[width=\linewidth]{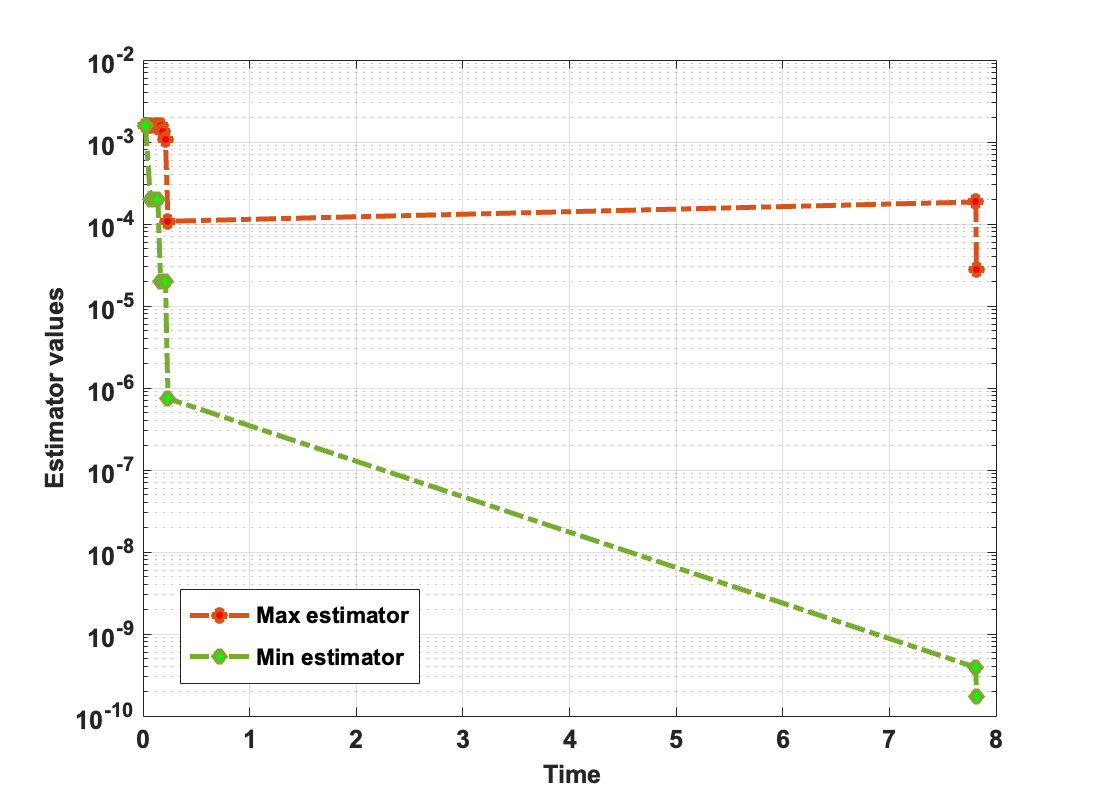}
		\caption{Estimator vs Time.}
		\label{esttime} 
	\end{subfigure}
	\caption{\textbf{Left:} Number of degrees of freedoms (\#DOF), and
		\textbf{Right:} Minimum and Maximum values of estimator $\mathpzc{R}_h$ on different time steps.}
	\label{dofandest} 
\end{figure}
The simulation was run for 1600 time steps using an adaptive mesh refinement strategy. As the fracture evolves, the mesh is refined, which increases the number of degrees of freedom (\#DOF) to ensure accurate resolution. The refinement strategy chooses the top $20\%$ of the cells with maximum value of $\mathpzc{R}_h$ and coarsens $5\%$ of the cells with minimum value of $\mathpzc{R}_h$. Further, \Cref{doft} shows the DOF vs time, while the maximum and minimum values over time is depicted by \Cref{esttime}. 
\section{Conclusion} \label{sec6:conc}
In this work, we have introduced and rigorously analyzed an adaptive, regularized phase-field framework for dynamic brittle fracture. We established, for the first time, a convergence proof for an adaptive finite element algorithm that couples a discrete wave equation with the minimization of the Ambrosio-Tortorelli energy functional. Our analysis demonstrates that the algorithm converges to the continuous elastodynamics solution and the minimizer of the total energy, with the total residual approaching zero. Furthermore, our numerical experiments validate the efficacy of the proposed adaptive mesh refinement strategy, successfully capturing complex phenomena such as crack branching during anti-plane shear propagation.

This framework lays a robust foundation for future investigations. The adaptive algorithm and its convergence analysis are well-suited for extension to more complex constitutive models, such as strain-limiting theories of elasticity \cite{rajagopal2007elasticity, manohar2025adaptive,manohar2025convergence}. Moreover, the proposed methodology provides a powerful tool for exploring challenging multiphysics problems, including fracture in thermo-mechanical or corrosive environments.

\section*{Acknowledgement}
This material is based upon work supported by the National Science Foundation under Grant No. 2316905. 

\bibliographystyle{siamplain}
\bibliography{RMSMP7.bib}
\end{document}